\documentclass[a4paper,12pt]{amsart}
\usepackage{lipsum}
\usepackage{amssymb,amsthm}
\usepackage[foot]{amsaddr}
\usepackage{amsmath}
\usepackage{tikz}
\usetikzlibrary{quotes,angles}
\usepackage{graphicx}
\usepackage{subcaption}
\usepackage[shortlabels]{enumitem}
\usetikzlibrary{arrows}
\usepackage{float}
\usepackage{graphicx}
\usepackage{subcaption}
\usepackage{hyperref}
\usepackage[section]{placeins}
\graphicspath{{Figures/}}
\usetikzlibrary{decorations.pathmorphing}
\tikzset{snake it/.style={decorate, decoration=snake}}
\usepackage{pgfplots}

\pgfplotsset{compat=1.18}
\makeatletter
\newcommand*{\rom}[1]{\expandafter\@slowromancap\romannumeral #1@}
\makeatother

\numberwithin{equation}{section}

\theoremstyle{plain}
\newtheorem{theorem}{Theorem}
\numberwithin{theorem}{section}
\newtheorem{proposition}[theorem]{Proposition}
\newtheorem{lemma}[theorem]{Lemma}
\newtheorem{corollary}[theorem]{Corollary}
\theoremstyle{definition}
\newtheorem{definition}[theorem]{Definition}
\theoremstyle{remark}
\newtheorem{remark}[theorem]{Remark}

\theoremstyle{remark}

\theoremstyle{remark}

\newtheorem{assumption}[theorem]{Assumption}

\newcommand{\al}{\alpha}
\newcommand{\bt}{\beta}

\newcommand{\s}{\mathcal S}

\newcommand{\smo}{\setminus \mathbf{0}}

\newcommand{\norm}[1]{\left\lVert#1\right\rVert}      
\newcommand{\abs}[1]{\left|#1\right|}                 
\newcommand{\paren}[1]{\left(#1\right)}               
\newcommand{\sparen}[1]{\left\{#1\right\}}      

\renewcommand{\d}{\,\mathrm{d}}  

\newcommand{\dd}{\mathrm{d}}  

\newcommand{\Cc}{\mathcal{C}}

\newcommand{\Dc}{\mathcal{D}}
\newcommand{\Ec}{\mathcal{E}}
\newcommand{\Fc}{\mathcal{F}}

\newcommand{\Sc}{\mathcal{S}}

\newcommand{\WF}{\mathrm{WF}} 

\newcommand{\wf}{\mathrm{WF}}                         

\newcommand{\va}{\mathbf{a}}

\newcommand{\partyf}[2]{\frac{\partial #2}{\partial y_{#1}}}

\newcommand{\vv}{{\mathbf{v}}}

\newcommand{\bpm}{\begin{pmatrix}}
\newcommand{\epm}{\end{pmatrix}}

\newcommand{\vx}{{\mathbf{x}}}

\newcommand{\vy}{{\mathbf{y}}}
\newcommand{\vz}{{\mathbf{z}}}

\newcommand{\vsig}{{\boldsymbol{\sigma}}} 

\newcommand{\rr}{{{\mathbb R}}}

\newcommand{\rn}{{{\mathbb R}^n}}

\newcommand{\be}{\begin{equation}}
\newcommand{\bea}{\begin{eqnarray}}
\newcommand{\eea}{\end{eqnarray}}
\newcommand{\bean}{\begin{eqnarray*}}
\newcommand{\eean}{\end{eqnarray*}}

\newcommand{\bel}[1]{\begin{equation}\label{#1}}
\newcommand{\ee}{\end{equation}}
\newcommand{\eel}[1]{{\label{#1}\end{equation}}}

\newcommand{\ssupp}{\text{ssupp}}

\usepackage{fullpage}

\usepackage{kbordermatrix}
\renewcommand{\kbldelim}{(}
\renewcommand{\kbrdelim}{)}

\usepackage{datetime}
\usetikzlibrary{quotes, angles}

\newcommand\irregularcircle[2]{
  \pgfextra {\pgfmathsetmacro\len{(#1)+rand*(#2)}}
  +(0:\len pt)
  \foreach \a in {10,20,...,350}{
    \pgfextra {\pgfmathsetmacro\len{(#1)+rand*(#2)}}
    -- +(\a:\len pt)
  } -- cycle
}

\title[short]{Microlocal analysis of a non-linear cone transform and applications to Compton camera imaging\\
}
\author{James W. Webber\textsuperscript{$\dagger$}}
\author{Sean Holman\textsuperscript{$\ddagger$}}
\address[James W. Webber (corresponding author)]{Department of Biomedical Engineering, Cleveland Clinic, USA}
\address[Sean Holman]{Department of Mathematics, The University of Manchester, Alan Turing Building, Oxford Road, Manchester M13 9PY, UK}
\email[A1,A2]{webberj5@ccf.org\textsuperscript{$\dagger$}, sean.holman@manchester.ac.uk\textsuperscript{$\ddagger$}}

\providecommand{\keywords}[1]
{
  \small	
  \textbf{\textit{Keywords---}} #1
}

\begin{document}

\begin{abstract}
We present a novel method to recover the source intensity, $f : \mathbb{R}^n \to \mathbb{R}$, and attenuation coefficient, $\mu : \mathbb{R}^n \to \mathbb{R}$, in Compton camera imaging. We apply a non-linear model, which accounts for ray attenuation. We show that the data, $h$, can be modeled $h = \mathcal{R}(f,\mu) = R(fg)$, where $g = \exp(-G\mu)$ models attenuation, $G$ is a (linear) divergent beam transform, and $R$ is a linear operator which defines the integrals of $fg$ over cones \cite{kuchment2017inversion}. Commonly in the literature, $\mu$ is set to zero, and the data $h = Rf$ is linear. We address the case when $\mu \neq 0$ and the transform is non-linear. To simplify the analysis, using similar ideas to \cite{kuchment2017inversion}, we first transform the data into weighted line integrals, $\tilde{h} = \mathcal{D}_k(f,\mu) = D_k(fg)$, where $D_k$ is a weighted ray transform. Assuming practically reasonable geometric conditions, we show that $\tilde{h} = \exp(-X_{w_1}\mu)X_{w_2}f$, where $X_w$ is a weighted X-ray transform, and the $w_i$ are smooth weights. After which, we use the theory of conormal distributions in a similar vein to \cite{palacios2018quantitative}, to describe the singularities of $\tilde{h}$. We show that there are artifacts in the reconstruction, and we quantify their strength using Sobolev spaces. We combine this theory with a geometric argument to recover the edges of $f$ and ultimately prove that $f$ and $\mu$ are unique to $h$. The recovery of $f$ is notably more stable than that of $\mu$, which we also discuss. To validate our theory, we present simulated reconstructions of $f$ and $\mu$ using the proposed method.

\end{abstract}
\maketitle

 - microlocal analysis, non-linear cone transforms

\section{Introduction}
In this paper, we present novel microlocal analysis of a non-linear cone transform, $\mathcal{R}(f,\mu)$, related to Compton camera imaging, where the quantities of interest are source density $f$ and attenuation $\mu$. Up to this point, to the best of our knowledge, only linear cone transforms have been addressed microlocally in this field. The attenuation of the rays creates the non-linearity, and we show this leads to additional singularities in the data that are not present under the linear formulation (i.e., when attenuation is neglected). We analyze specifically what happens when we apply a linear reconstruction operator to the non-linear data, and our first aim is to recover $f$. We show that there are artifacts in the reconstruction and quantify their strength on Sobolev scale, and we derive methods to suppress or remove the artifacts. Once $f$ is recovered, we provide conditions so that $\mu$ is unique to the data. Specifically, we show that $\mu$ is recoverable, although unstably due to limited data effects, akin to those typically seen in conventional limited-angle tomography \cite{frikel2013characterization,krishnan2014microlocal}.

The literature covers extensively linear cone and V-line (or broken-ray) transforms \cite{terzioglu2018compton, kuchment2016three, kuchment2017inversion, moon2025inversion, terzioglu2019some, AmbartsoumianLatifi-Vline2019, vline, morvidone2010v, ambartsoumian2012inversion, terzioglu2020exact, zhang2020recovery}. In \cite{kuchment2017inversion}, the authors consider a weighted cone transform, and show that it is essentially equivalent to a type of weighted divergent beam transform, where the end point of the divergent beam is the same as the apex of the cone. They derive inversion formulae for the divergent beam transform under Tuy's condition, and combine these ideas to invert the cone transform. Relations to the classical Radon transform are also later presented. In a similar way, the authors in \cite{terzioglu2020exact} derive a closed-form formula which relates cone integral data to plane integrals in $\mathbb{R}^3$, and they use this and a variation of Tuy's condition (denoted the ``Compton admissability condition") to invert the cone transform when the cone vertices are constrained to a curve.

In \cite{moon2025inversion}, the authors consider an attenuated cone transform with radial weight and where the attenuation coefficient ($\mu$) is constant. The cone vertices are contrained to either a plane or cylinder. They decompose the operator into a spherical section transform and attenuated ray transform, and then derive closed-form inversion formulae for each transform separately combining these together to invert the cone transform. While attenuation is included in the model, $\mu$ is likely not constant in many applications of interest. We consider a different case in which $\mu$ varies and includes discontinuities.

Microlocal analysis of a linear cone transform, $C_k$, is covered in \cite{terzioglu2019some}. There the authors assume they have access to all integrals of a function $f : \mathbb{R}^n \to \mathbb{R}$ over cones in $\mathbb{R}^n$, where $n$ is the dimension, and thus the data has dimension $2n$ (an overdetermined problem). They analyze $C_k$ as a Fourier Integral Operator (FIO) and go on to show that $C_k^*C_k$ is a pseudodifferential operator ($\Psi$DO) order $1-n$, which means there are no unwanted singularities (artifacts) added in a filtered backprojection type reconstruction.

In \cite{zhang2020recovery}, the authors analyze microlocally a weighted cone transform $I_\kappa$ in $\mathbb{R}^3$, where $\kappa$ is a smooth weight. The cone vertices are limited to a smooth 2-D surface. Under stated conditions on $\kappa$, they show that $I_\kappa^*I_\kappa$ acts as an elliptic $\Psi$DO near visible singularities (i.e., those which are detectable in the data). Since $I_\kappa^*I_\kappa$ is elliptic, this means one can stably recover the visible singularities, and the authors prove this. They later show that their analysis can be applied when the cone vertices are restricted to a 1-D curve and the cone opening angle remains fixed. In our case, standard linear FIO theory does not apply since our integral weights depend on $\mu$, which can have singularities.

We introduce a novel non-linear cone transform $\mathcal{R}(f,\mu)$, where the cone vertices, similar to \cite{zhang2020recovery}, are constrained to lie on a smooth $(n-1)$-dimensional surface. More precisely, the transform is linear in $f$ and non-linear in $\mu$. The cone orientation and opening angle is unrestricted. We assume that $f$ and $\mu$ are the finite sums of characteristic type functions with disjoint supports, allowing for more general modeling of $\mu$ than previously considered. This is important, as unless the attenuating material is very small in size, there will be significant signal loss due to ray attenuation. While attenuation is addressed somewhat in \cite{moon2025inversion}, constant $\mu$ does not account for smooth variations or possible singularities. We address this here.

To simplify the analysis, using similar ideas to that of \cite{kuchment2017inversion,terzioglu2020exact}, we first convert $\mathcal{R}(f,\mu)$ into weighted ray integral data $\mathcal{D}_k(f,\mu)$, where $\mathcal{D}_k$ is a weighted divergent beam transform. This transform is similar to that considered in the SPECT identification problem \cite{holman2020spect,SeanHolman,natterer1981identification,solmon1995identification,stefanov2014identification}. In \cite{SeanHolman}, the authors show that $f$ and $\mu$ are unique to the data under certain geometric conditions. For example, they assume that the singularities of $\mu$ are visible although this is not generally true in our application. Further, in \cite{SeanHolman}, only 2-D imaging is considered. We address the $n$-D case. Methods of \cite{SeanHolman} could reproduce some of the results of this manuscript, notably the determination of the boundary of $\Omega$ considered in section \ref{sec:bdet}, but not the precise determination of Sobolev order of reconstructed singularities completed in sections \ref{sec:bdet} and \ref{sing_recon} nor the unique recovery of $\mu$ as shown in section \ref{sec:sim}.

To begin our analysis, we first establish our geometric assumptions, which are practically reasonable.  One of our key assumptions is that $\text{supp}(\mu)$ and the convex hull of $\text{supp}(f)$ are disjoint, which means there is no attenuating material in the non-convex ``grooves" of the gamma ray source. Under this condition, we show that $\tilde{h} = \mathcal{D}_k(f,\mu) = \exp(-X_{w_1}\mu)X_{w_2}f$, where $X_w$ is a weighted X-ray transform, and the $w_i$ are smooth weights. This is essentially a product of distributions with singularities on intersecting smooth manifolds, which have been studied previously in, e.g., \cite{palacios2018quantitative}, and by Hormander in \cite[Theorem 8.2.10]{hormanderI}. Using these ideas and the theory of conormal distributions, we describe the singularities of $\tilde{h}$, and aim first to recover $f$. We show that the non-linearities in the data lead to artifacts in the reconstruction. To address this, we use a geometric argument to discern the true edges of $f$ from the artifacts, and present a method to recover $f$. We later expand the theory of Natterer \cite[Theorem 3.3]{natterer} to derive uniqueness results for $\mu$, and ultimately provide conditions so that $f$ and $\mu$ are unique to $\mathcal{R}(f,\mu)$. To validate this theory, we present simulated image reconstructions in the case when $n = 2$.

The remainder of this paper is organized as follows. In section \ref{sect:defns}, we state some definitions and preliminary theorems that we will need later. In section \ref{main}, we state our main assumptions and show how to recover $f$ from $\mathcal{R}(f,\mu)$. In section \ref{sec:sim}, we generalize \cite[Theorem 3.3]{natterer}, and use the theorem to derive uniqueness results for $\mu$. Finally, in section \ref{results}, we present simulated image reconstructions with added noise, and go through a worked example which showcases the key steps of our method.

\section{Definitions}\label{sect:defns} 

In this section, we review some theory from microlocal analysis which will be used in our theorems. For readers who prefer to skip straight to the main material on the non-linear cone transform, it is possible to skip this section and refer back as the definitions and results are needed.

We first provide some
notation and definitions.  Let $X$ and $Y$ be open subsets of
{$\mathbb{R}^{n_X}$ and $\mathbb{R}^{n_Y}$, respectively.}  Let $\Dc(X)$ be the space of smooth functions compactly
supported on $X$ with the standard topology and let $\mathcal{D}'(X)$
denote its dual space, the vector space of distributions on $X$. This notation should not be confused with the nonlinear operator introduced in \eqref{trans_data}. Let
$\Ec(X)$ be the space of all smooth functions on $X$ with the standard
topology and let $\mathcal{E}'(X)$ denote its dual space, the vector
space of distributions with compact support contained in $X$. Finally,
let $\Sc(\rn)$ be the space of Schwartz functions, that are rapidly
decreasing at $\infty$ along with all derivatives and $\mathcal{S}'(\mathbb{R}^n)$ the corresponding space of tempered distributions. See \cite{Rudin:FA}
for more information. 


We now list some notation conventions that will be used throughout this paper:
\begin{enumerate}
\item For a function $f$ in the Schwartz space $\Sc(\mathbb{R}^{n_X})$, we {write
\[
\mathcal{F}f(\xi) = \int_{\mathbb{R}^{n_X}} e^{-i x\cdot \xi}f(x)\ \dd x,\quad \mathcal{F}^{-1}f(x) = \frac{1}{(2 \pi)^{n_x}}\int_{\mathbb{R}^{n_X}} e^{i x\cdot \xi}f(\xi)\ \dd \xi
\]
for} the Fourier transform and inverse Fourier transform of $f$,
respectively, {and extend these in the usual way to tempered distributions $\Sc'(\mathbb{R}^{n_x})$} (see \cite[Definition 7.1.1]{hormanderI}). We will use the notation $\hat{f} = \mathcal{F}f$. 

\item We use the standard multi-index notation: if
$\al=(\al_1,\al_2,\dots,\al_n)\in \sparen{0,1,2,\dots}^{n_X}$
is a multi-index and $f$ is a function on $\mathbb{R}^{n_X}$, then
\[\partial^\al f=\paren{\frac{\partial}{\partial
x_1}}^{\al_1}\paren{\frac{\partial}{\partial
x_2}}^{\al_2}\cdots\paren{\frac{\partial}{\partial x_{n_X}}}^{\al_{n_X}}
f.\] If $f$ is a function of $(\vy,\vx,\vsig)$ then $\partial^\al_\vy f$ and $\partial^\al_\vsig f$ are defined similarly. We use the notation $|\alpha| = \sum_{i=1}^n \alpha_i$.

\item \label{item:T*Xident} We identify the cotangent
spaces of Euclidean spaces with the underlying Euclidean spaces. For example, the cotangent space, 
$T^*(X)$, of $X$ is identified with $X\times \mathbb{R}^{n_X}$. If $\Phi$ is a function of $(\vy,\vx,\vsig)\in Y\times X\times \rr^N$,
then we define $\dd_{\vy} \Phi = \paren{\partyf{1}{\Phi},
\partyf{2}{\Phi}, \cdots, \partyf{{n_X}}{\Phi} }$, and $\dd_\vx\Phi$ and $
\dd_{\vsig} \Phi $ are defined similarly. Identifying the cotangent space with the Euclidean space as mentioned above, we let $\dd\Phi =
\paren{\dd_{\vy} \Phi, \dd_{\vx} \Phi,\dd_{\vsig} \Phi}$.


\end{enumerate}

\noindent The singularities of a function and the directions in which they occur
are described by the wavefront set \cite[page
16]{duistermaat1996fourier}, which we now define.
\begin{definition}
\label{WF} Let $X$ be an open subset of $\mathbb{R}^{n_X}$ and let $f$ be a
distribution in $\mathcal{D}'(X)$.  Let $(\vx_0,\xi_0)\in X\times
(\rr^{n_X}\smo)$.  Then $f$ is \emph{smooth at $\vx_0$ in direction $\xi_0$} if
there exists a neighborhood $U$ of $\vx_0$ and $V$ of $\xi_0$ such
that for every $\Phi\in \Dc(U)$ and $N\in\mathbb{R}$ there exists a
constant $C_N$ such that for all $\xi\in V$ {and $\lambda >1$},
\begin{equation}
\left|\Fc(\Phi f)(\lambda\xi)\right|\leq C_N(1+\abs{\lambda})^{-N}.
\end{equation}
The pair $(\vx_0,\xi_0)\in X\times
(\rr^{n_X}\smo)$ is in the \emph{wavefront set} of $f$, $\wf(f)$, if
$f$ is not smooth at $\vx_0$ in direction $\xi_0$. We define for fixed $\vx_0$, $\WF_{\vx_0}(f) = \{\xi : (\vx_0,\xi) \in \WF(f)\}$, and we define the singular support of $f$, $\text{ssupp}(f)$, as the natural projection of $\WF(f)$ onto $X$.
\end{definition}

Intuitively, the elements 
$(\vx_0,\xi_0)\in \WF(f)$ are the point-normal vector pairs at
which $f$ has singularities; $\vx_0$ is the location of the 
  singularity, and  $\xi_0$ is the direction in which the
  singularity occurs.
  A
 geometric example of the wavefront set is given by the characteristic function $f$ of a domain $\Omega \subset \mathbb{R}^{n_X}$ with smooth boundary, which is $1$ on $\Omega$ and $0$ off of $\Omega$. Then the wavefront set is
\[
\WF(f) = \{(\vx,t\vv) \ : t \neq 0, \ \vx \in \partial \Omega, \ \mbox{$\vv$ is orthogonal to $\partial \Omega$ at $\vx$}\}.
\]
In other words, the wavefront set is the set of points in the boundary of $\Omega$ together with the nonzero normal vectors to the boundary. {The set of normals of the surface $\partial \Omega$ is a subset of the cotangent bundle $T^*X$, and here we are using the identification of $T^*X$ with $X \times \mathbb{R}^{n_X}$ mentioned above in point (\ref{item:T*Xident}).} The wavefront set is an important consideration in imaging since elements of the wavefront set will correspond to sharp features of an image.




The wavefront set of a distribution on $X$ is normally defined as a
subset of the cotangent bundle $T^*(X)$ so it is invariant under
diffeomorphisms, but we do not need this invariance, so we will
continue to identify $T^*(X) = X \times \rr^{n_X}$ and consider $\WF(f)$ as
a subset of $X\times (\rr^{n_X}\smo)$.


\begin{definition}[{\cite[Definition
        21.2.15]{hormanderIII}}] \label{phasedef}
A function $\Phi=\Phi(\vy,\vx,\vsig)\in
\Ec(Y\times X\times(\mathbb{R}^N\smo))$ is a \emph{phase
function} if $\Phi(\vy,\vx,\lambda\vsig)=\lambda\Phi(\vy,\vx,\vsig)$, $\forall
\lambda>0$ and $\mathrm{d}\Phi$ is nowhere zero. The
\emph{critical set of $\Phi$} is
\[\Sigma_\Phi=\{(\vy,\vx,\vsig)\in Y\times X\times(\mathbb{R}^N\smo)
: \dd_{\vsig}\Phi=0\}.\] 
 A phase function is
\emph{clean} if the critical set $\Sigma_\Phi$ is a smooth manifold {with tangent space defined {by} the kernel of $\mathrm{d}\,(\mathrm{d}_\sigma\Phi)$ on $\Sigma_\Phi$. Here, the derivative $\mathrm{d}$ is applied component-wise to the vector-valued function $\mathrm{d}_\sigma\Phi$. So, $\mathrm{d}\,(\mathrm{d}_\sigma\Phi)$ is treated as a Jacobian matrix of dimensions $N\times ({n_Y + n_X}+N)$.}
\end{definition}

\begin{definition}[{\cite[Definition 21.2.15]{hormanderIII} and
      \cite[section 25.2]{hormander}}]\label{def:canon} Let $X$ and
$Y$ be open subsets of $\rn$. Let $\Phi\in \Ec\paren{Y \times X \times
{\rr}^N}$ be a clean phase function.  In addition, we assume that
$\Phi$ is \emph{nondegenerate} in the following sense:
\[\text{$\dd_{\vy}\Phi$ and $\dd_{\vx}\Phi$ are never zero on
$\Sigma_{\Phi}$.}\]
  The
\emph{canonical relation parametrized by $\Phi$} is defined as
\begin{equation}\label{def:Cgenl} \begin{aligned} \Cc=&\sparen{
\paren{\paren{\vy,\dd_{\vy}\Phi(\vy,\vx,\vsig)};\paren{\vx,-\dd_{\vx}\Phi(\vy,\vx,\vsig)}}:(\vy,\vx,\vsig)\in
\Sigma_{\Phi}}{.}
\end{aligned}
\end{equation}
\end{definition}

 \begin{definition}[{\cite[Definition 7.8.1]{hormanderI}}] \label{ellip}We define
 $S^m(Y \times X, \mathbb{R}^N)$ to be the
set of $a\in \Ec(Y\times X\times \mathbb{R}^N)$ such that for every
compact set $K\subset Y\times X$ and all multi--indices $\al,
\bt, \gamma$ the bound
\[
\left|\partial^{\gamma}_{\vy}\partial^{\bt}_{\vx}\partial^{\al}_{\vsig}a(\vy,\vx,\vsig)\right|\leq
C_{K,\al,\bt,\gamma}(1+\norm{\vsig})^{m-|\al|},\ \ \ (\vy,\vx)\in K,\
\vsig\in\mathbb{R}^N,
\]
holds for some constant $C_{K,\al,\bt,\gamma}>0$. The elements of $S^m$ are called \emph{symbols} of order $m$.  Note
that this symbol class is  sometimes denoted $S^m_{1,0}$.
\end{definition}

\begin{definition}\label{FIOdef}
Let $X$ and $Y$ be open subsets of {$\mathbb{R}^{n_X}$ and $\mathbb{R}^{n_Y}$, respectively.} {Let an operator $A :
\Dc(X)\to \mathcal{D}'(Y)$ be defined by the distribution kernel
$K_A\in \mathcal{D}'(Y\times X)$, in the sense that
$Af(\vy)=\int_{X}K_A(\vy,\vx)f(\vx)\mathrm{d}\vx$. Then we call $K_A$
the \emph{Schwartz kernel} of $A$}. A \emph{Fourier
integral operator (FIO)} of order $\mu = m + N/2 - (n_X+n_Y)/4$ is an operator
$A:\Dc(X)\to \mathcal{D}'(Y)$ with Schwartz kernel given by an
oscillatory integral of the form
\begin{equation} \label{oscint}
K_A(\vy,\vx)=\int_{\mathbb{R}^N}
e^{i\Phi(\vy,\vx,\vsig)}a(\vy,\vx,\vsig) \mathrm{d}\vsig,
\end{equation}
where $\Phi$ is a clean nondegenerate phase function and $a$ is a
symbol in $S^m(Y \times X , \mathbb{R}^N)$. The \emph{canonical
relation of $A$} is the canonical relation $\mathcal{C}$ of $\Phi$ defined in
\eqref{def:Cgenl}. 
An FIO is called a \emph{pseudodifferential operator} if {$X = Y$ and} its canonical relation $\Cc$ is contained in the diagonal, i.e.,
$\Cc \subset \Delta := \{ (\vx,\xi;\vx,\xi)\}$.
\end{definition}


Fourier integral operators are defined in \cite{hormander} more generally as operators with Schwartz kernel locally given by expressions of the form \eqref{oscint} where the phase functions each give rise to pieces of the same global cannonical relation. However, the local expression \eqref{oscint} is sufficient for our purposes.


We have the definition of Sobolev spaces from \cite[page 200]{natterer}.

\begin{definition}\label{sobo_spaces}
We define the Sobolev space order $\al$,
\begin{equation}
H^\al(\mathbb{R}^n) = \{f\in \mathcal{S}'(\mathbb{R}^n) : (1 + |\xi|^2)^{\al/2} \hat{f} \in L^2(\mathbb{R}^n)\}
\end{equation}
with norm 
\begin{equation}\label{HRnnorm}
\|f\|_{H^{\al}(\mathbb{R}^n)} = \|(1 + |\xi|^2)^{\al/2} \hat{f}\|_{L^2(\mathbb{R}^n)}.
\end{equation}
\end{definition}

\noindent We now define local and microlocal Sobolev regularity and wavefront sets \cite{Q1993sing}.

\begin{definition} \label{def:SobolevWF}
A distribution $g$ is in $H^{\al}$ locally near a point $\vx_0$ if and only if there exists a cut-off function $\varphi \in C_c^{\infty}(\mathbb{R}^n)$ with $\varphi(\vx_0) \neq 0$ such that $\varphi g \in H^{\al}(\mathbb{R}^n)$. The distribution $g$ is in $H^{\al}$ microlocally near $(\vx_0,\xi_0)$ if and only if there is a cut-off function $\varphi \in C^{\infty}_c(\mathbb{R}^n)$ with $\varphi(\vx_0) \neq 0$ and function $u(\xi)$ homogeneous of degree zero and smooth on $\mathbb{R}^n\backslash \{0\}$ with $u(\xi_0)\neq 0$ such that $(1+|\xi|^2)^{\al/2}u(\xi)\mathcal{F}(\varphi g)(\xi) \in L^2(\mathbb{R}^n)$. We define the Sobolev wavefront set
\begin{equation}
\WF^{\al}(g) = \left\{ (\vx,\xi) \in \WF(g) : g\ \text{is not in}\ H^{\al}\ \text{microlocally near}\ (\vx, \xi) \right\}
\end{equation}
and $\text{ssupp}^{\al}(g)$ as the natural projection of $\WF^{\al}(g)$ onto $X$.
\end{definition}

We define the distribution $x_+ = 0 $ for $x < 0$ and $x_+ = x$ when $x\geq 0$. Similarly we define $x_- = (-x)_+$. From \cite{franssens2014multiplication}, we have the theorem which defines the Fourier transform of such distributions.
\begin{theorem}
\label{thm_dist}
Let $z \in \mathbb{C}$. We have
\begin{equation}
\label{equ_dist}
\mathcal{F} (x_\pm^{-(z+1)})(\lambda) = c(z)(\lambda \mp i0)^z,\ z\not=0,1,2,\dots,
\end{equation}
where here $\lambda$ is dual to $x$ and $c(z) = (\pm 2\pi i )^z\Gamma(z)$ where $\Gamma$ is the Gamma Function. Further, the distributions on the right side of \eqref{equ_dist} can be written
$$(\lambda \pm i0)^z = \lambda_+^z + e^{\pm i\pi z } \lambda_-^z = |\lambda|^z[H(\lambda) + e^{\pm i\pi z } H(-\lambda)],$$
where $H$ is the Heaviside function. 
\end{theorem}
Thus, for real $z$, the Fourier transform of $x_\pm^{z}$ decays at rate $z+1$ as $|\lambda| \to \infty$.

\noindent We now review some preliminaries on conormal distributions.

\begin{definition}[\cite{palacios2018quantitative,greenleaf1993recovering}]
\label{def_conormal}
Let $\Phi(\vx,\vsig)$ be a phase function, let $a(\vx,\vsig)$ be a symbol order $\nu$, and let $\Gamma$ be a smooth embedded submanifold of codimension $N$ in $X$. Suppose the normal bundle of $\Gamma$ is parametrized by $\Phi$, which means $N^*\Gamma = \{(\vx, \mathrm{d}_\vx\Phi(\vx,\vsig)) \in T^* X : \vx \in X, \mathrm{d}_\vsig\Phi(\vx,\vsig) = 0\}$. Then, we denote $I^\beta(N^*\Gamma) = I^{\beta + n/4 -N/2}(\Gamma) = I^{\nu}(\Gamma)$ to be the space of conormal distributions order $\beta$ on $N^*\Gamma$. The elements $u \in I^\beta(N^*\Gamma)$  take the form
\begin{equation}
u = \int_{\mathbb{R}^N}a(\vx,\vsig) e^{i \Phi(\vx,\vsig)} \mathrm{d}\vsig,
\end{equation}
where $a \in S^\nu(X, \mathbb{R}^N)$, $\nu = \beta + n/4 -N/2$ and one has $u \in H^{\alpha}(\mathbb{R}^{n_X})$ for $\alpha < - \beta -n/4$.
\end{definition}

\begin{proposition}
\label{prop_sqrt}
Let $s - \phi(\vv) = 0$ be the defining equation of a smooth $(n-1)$-dimensional manifold $\Gamma$, in $\mathbb{R}^n$, where $\phi$ is smooth, $s\in \mathbb{R}$ and $\vv \in \mathbb{R}^{n-1}$. Then, for $z >0$ and any smooth function $u$,
\begin{equation}
\label{dist_order_0}
u(s,\vv)(s - \phi(\vv))^{z}_\pm \in I^{-z-1/2-n/4}(N^*\Gamma).
\end{equation}
\end{proposition}
\begin{proof}
By \eqref{equ_dist}, we have
\begin{equation}
u(s,\vv)(s - \phi(\vv))^{z}_\pm = \frac{2\pi u(s,\vv)}{c(z)}\int_{\mathbb{R}}(\lambda \mp i0)^{-(1+z)} e^{i \lambda (s- \phi(\vv))} \mathrm{d}\lambda,
\end{equation}
where $c(z)$ is as in \eqref{equ_dist}. Now, let $\varphi(\lambda)$ be a smooth cutoff centered on zero which is equal to 1 in a small neighborhood of zero. Then, 
\begin{equation}
\begin{split}
u(s,v)(s - \phi(\vv))^{z}_\pm &= \frac{2\pi u(s,\vv)}{c(z)}\Bigg [ \int_{\mathbb{R}}\varphi(\lambda)(\lambda \mp i0)^{-(1+z)} e^{i \lambda (s- \phi(\vv))} \mathrm{d}\lambda\\
& \hskip1cm + \int_{\mathbb{R}}(1- \varphi(\lambda)) (\lambda \mp i0)^{-(1+z)} e^{i \lambda (s- \phi(\vv))} \mathrm{d}\lambda \Bigg]\\
&= \frac{2\pi u(s,\vv)}{c(z)} \nu(s,\vv) + \int_{\mathbb{R}}a(s,\vv,\lambda) e^{i \lambda (s- \phi(\vv))} \mathrm{d}\lambda.
\end{split}
\end{equation}
Here, $\nu$ is smooth by the Paley-Weiner theorem \cite[Theorem 7.3.1]{hormanderI}, and $a$ is a symbol order $\mu = -(1+z)$, noting that we have localized away from the singularity in $(\lambda \mp i0)$ at zero using the smooth cutoff. We now calculate the order of the distribution ($\beta$) using $\mu = \beta + n/4 - N/2$, where $N = 1$ is the dimension of the variable $\lambda$. This yields $\beta = -z-1/2-n/4$.
\end{proof}

The normal bundle $N^* \Gamma$, as used in Definition \ref{def_conormal}, is an example of a Lagrangian manifold \cite[Chapter XXI]{hormanderIII}. We will use the so-called $I^{p,l}$ classes of distributions associated to a pair of cleanly intersecting Lagrangians. The next definition applies to general Lagrangian manifolds, but we will only require the case of cleanly intersecting normal bundles.

\begin{definition}[\cite{greenleaf1993recovering}]
\label{pair_lag}
Let $\Lambda_1,\Lambda_2 \subset T^*X\backslash 0$ be a pair of cleanly intersecting Lagrangians. Then, associated to the pair $(\Lambda_1, \Lambda_2)$ is a class of Lagrangian distributions $u\in I^{p,l}(\Lambda_1, \Lambda_2)$, indexed by $p,l\in \mathbb{R}$, which satisfy $\WF(u) \subset \Lambda_1 \cup \Lambda_2$. Microlocally, away from $\Lambda_1 \cap \Lambda_2$,
$$I^{p,l}(\Lambda_1,\Lambda_2) \subset I^{p+l}(\Lambda_1 \backslash \Lambda_2),\ \ \text{and}\ \ I^{p,l}(\Lambda_1,\Lambda_2) \subset I^p(\Lambda_2).$$
Let $Y_2 \subset Y_1 \subset X$ be smooth manifolds with $\text{codim}_X(Y_1) = d_1$ and $\text{codim}_X(Y_2) = d_1 + d_2$. Then $N^*Y_1$ and $N^*Y_2$ intersect cleanly in codimension $d_2$. The space of distributions on $X$ conormal to the pair $(Y_1,Y_2)$ of orders $\nu, \nu'$ is
\begin{equation}\label{pair_lag_e}
\begin{split}
I^{\nu,\nu'}(Y_1,Y_2) &= I^{\nu + \nu' + \frac{d_1 +d_2}{2}-\frac{n}{4},-\frac{d_2}{2}-\nu'}(N^*Y_1, N^* Y_2)\\
&= I^{\nu + \frac{d_1}{2}-\frac{n}{4},\nu'+\frac{d_2}{2}}(N^*Y_2, N^* Y_1)
\end{split}
\end{equation}
\end{definition}

Formulae for distributions in $I^{p,l}(\Lambda_1,\Lambda_2)$ are given in \cite{greenleaf1993recovering}, but their explicit form is not important for our purposes. We now have the lemma from \cite[Lemma 1.1]{greenleaf1993recovering}.
\begin{lemma}
\label{lem_green}
Let $Y, Z \subset X$ be submanifolds, where $Y$ and $Z$ intersect transversally. Then, 
$$I^\beta(Y)\cdot I^{\beta'}(Z) \subset I^{\beta,\beta'}(Y,Y \cap Z)  + I^{\beta',\beta}(Z,Y \cap Z).$$
\end{lemma}

\section{A non-linear cone transform}
\label{main}
In this section, we introduce our transform and present our first main microlocal theorems. We then go on to use this theory to reconstruct  a gamma ray source intensity from Compton CT data. 

Let us define the cone
$$C(\vv, \omega, h) = \{ \vx \in \mathbb{R}^n : (\vx - \vv) \cdot \omega - h|\vx - \vv| = 0\},$$
where $\omega \in S^{n-1}$ is the orientation of the central axis and $h \in (-1,1)$ determines the cone opening angle, $2\alpha$ ($\cos\alpha = h$). Here, $\vv\in \s$ is the cone vertex, where $\s$ is an $(n-1)$-dimensional manifold in $\mathbb{R}^n$. For example, $\s$ could be a sphere or a plane. When $n = 2$, $C$ defines a V-line or broken-ray \cite{AmbartsoumianLatifi-Vline2019}.

\begin{figure}[!h]
\centering
\begin{tikzpicture}[scale=0.8]
\draw[fill=green,rounded corners=1mm] (0,0) \irregularcircle{0.8cm}{0.5mm};
\draw[fill=cyan,rounded corners=1mm] (-2.5,0) \irregularcircle{0.8cm}{0.5mm};
\draw[fill=cyan,rounded corners=1mm] (-0.5,-2.5) \irregularcircle{0.8cm}{0.5mm};
\draw[fill=cyan,rounded corners=1mm] (1.5,2.5) \irregularcircle{0.8cm}{0.5mm};
\node at (-4.1,1.5) {$\s$};
\node at (-0.3,-0.3) {$f$};
\node at (-2.5,0) {$\mu$};
\node at (-0.5,-2.5) {$\mu$};
\node at (1.5,2.5) {$\mu$};
\draw [->,line width=1pt] (-6,0)--(6,0)node[right] {$x_1$};
\draw [->,line width=1pt] (0,-6)--(0,6)node[right] {$x_2$};
\draw [thick] (0,0) circle (4);
\draw  (1,5)node[right]{$C(\vv,\omega,h)$}--(2.83,-2.83)node[right]{$\vv$};
\draw  (-4.5,0)--(2.83,-2.83);
\draw [dashed,->]  (-1.75,2.5)node[left]{$\omega$}--(2.83,-2.83);
\coordinate (D) at (-1.75,2.5);
\coordinate (w) at (2.83,-2.83);
\coordinate (a) at (1,5);
\draw pic[draw=orange, <->,"$\alpha$", angle eccentricity=2] {angle = a--w--D};
\end{tikzpicture}
\caption{Example ECST geometry when $n=2$, $\s = S^{1}$ and $B = \{|\vx| < 1\}$. Here, the cyan domains represent where $\mu$ is supported, and the green domain is $\text{supp}(f) = \Omega$. In this example, the practical goal would be to identify the radioactive material $f$ surrounded by clutter $\mu$.}
\label{fig1}
\end{figure}
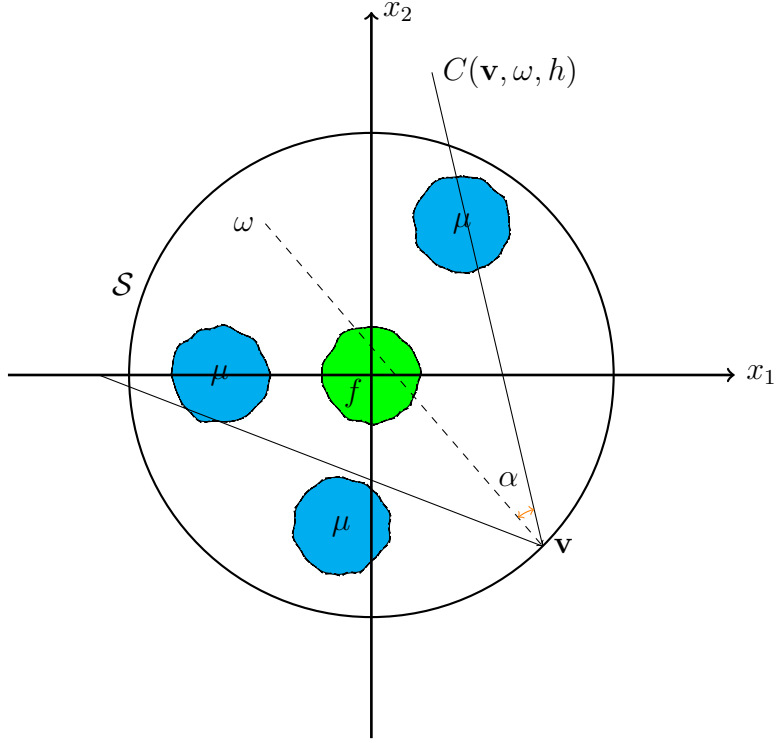

Let $B \subset \mathbb{R}^n$ be an open, simply connected domain with $B\cap \s = \emptyset$ which represents the region of interest (ROI), e.g., $B$ could be an open ball disjoint from $\s$. Let $f = I_0 \chi_\Omega$ represent a gamma ray source, where $I_0$ is the (constant) intensity and $\Omega \subset B$ is a compact domain in $B$ with smooth boundary, $\partial \Omega$. Further, let
\begin{equation}\label{eq:muform}
\mu = \sum_{k = 0}^m \Phi_k \chi_{\Omega_k}
\end{equation}
be an attenuation coefficient, where the $\Phi_k > 0$ are smooth functions and the $\Omega_k \subset B$ are compact domains with smooth boundaries such that $\Omega_k \cap \Omega = \Omega_k \cap \Omega_j = \emptyset$ for every $k$ and $j \neq k$. That is, the supports of $f$ and $\mu$ are assumed to be disjoint. 

Then, we consider the transform
\begin{equation}\label{orig_data}
\mathcal{R}(f, \mu)(\vv, \omega, h) = \int_{\vx\in C(\vv, \omega, h)} f(\vx) e^{ -G\mu(\vx, \vv) } \mathrm{d}S,
\end{equation}
where $\mathrm{d}S$ is the surface measure on $C$, and 
\begin{equation}
\begin{split}
G\mu(\vx, \vv) &= |\vx - \vv| \int_0^1\mu(t\vx + (1-t)\vv)\mathrm{d}t \\
\end{split}
\end{equation}
defines the integral of $\mu$ over the line segment from $\vv$ to $\vx$. Similarly, we define the linear analog $Rf(\vv, \omega, h) = \mathcal{R}(f, 0)(\vv, \omega, h)$. If we fix $\vv$, then we can write $\mathcal{R}(f, \mu) = R(fg)$, where $fg = f e^{ -G\mu(\cdot, \vv) }$. The operator, $\mathcal{R}$, models the Compton scatter intensity in emission Compton scatter tomography (ECST) applications, where a  Compton camera is used to detect the presence of a radioactive source, e.g., in nuclear decommissioning \cite{yao2022rapid}. An example geometry is shown in figure \ref{fig1}. $\mathcal{R}$ is linear in $f$ and non-linear in $\mu$ due to the exponential term, which models the attenuation of gamma rays traveling towards the Compton camera detector. Our main goal is to recover $f$ from $\mathcal{R}(f,\mu)$. To do this, we will show that the singularities of $f$ are uniquely determined. Later, in section \ref{sec:sim}, we will prove that $\mu$ is also determined, under reasonable assumptions on $\text{supp}(\mu)$.

First, we have the following result which shows how one can transform the cone integral data into ray (or line) integrals.

\begin{theorem}
\label{thm_1}
Let
\[
Rf_{lm}(\vv,h) = \int_{S^{n-1}}Rf(\vv, \omega, h) Y_{lm}(\omega) \d \omega \mbox{ and } f_{lm}(r) = \int_{S^{n-1}} f(r\omega)  Y_{lm}(\omega) \d \omega,
\]
where $Y_{lm}$ is a spherical harmonic. Then,
\begin{equation}
Rf_{lm}(\vv, h) = c_n(1 - h^2)^{(n-2)/2} C^{\lambda}_l(h) \int_0^\infty \rho^{n-2} \tilde{f}_{lm}(\rho) \mathrm{d}\rho,
\end{equation}
where $\tilde{f}(\vx) = f(\vv +\rho\vx)$, the $C^{\lambda}_l(h)$ are Gegenbaur polynomials \cite{Q1983-rotation} with degree $l$, $\lambda = (n-2)/2$, and $c_n = w_n/C^{\lambda}_l(1)$, where $w_n$ is the surface area of $S^{n-1}$.
\end{theorem}
\begin{proof}
Without loss of generality, let $\vv = 0$. This does not sacrifice generality as we can simply recenter at $\vv$ after the result is proven for $\vv = 0$. 

Let $Rf(\omega, h) = Rf(0,\omega,h)$. Then, we have
\begin{equation}
\begin{split}
Rf(\omega, h)  &= \int_{\mathbb{R}^n} |\nabla_\vx\Phi| \delta( \vx\cdot \omega - h|\vx| )f(\vx) \mathrm{d}\vx\\
&= \sqrt{1 - h^2} \int_{\xi \in S^{n-1}}\int_0^\infty \rho^{n-1} \delta( \rho(\xi\cdot \omega) - \rho h)f(\rho \xi) \d \rho \d \xi\\
&= \sqrt{1 - h^2} \int_{\xi \in S^{n-1}} \int_0^\infty\rho^{n-2}  f(\rho \xi) \delta(h - \xi \cdot \omega) \mathrm{d}\rho \mathrm{d}\xi,
\end{split}
\end{equation}
where $\Phi(\vx,\omega,h) = \vx\cdot \omega - h|\vx|$ and we have used the fact $|\nabla_\vx \Phi| = \sqrt{1 - h^2}$ when restricted to the set where $\Phi = 0$. 
Thus,
\begin{equation}
\begin{split}
R(f_{lm}Y_{lm})(\omega, h)  &= \sqrt{1 - h^2}\left[\int_0^\infty\rho^{n-2}  f_{lm}(\rho) \mathrm{d}\rho \right] \int_{\xi \in S^{n-1}} \delta(h - \xi \cdot \omega) Y_{lm}(\xi) \mathrm{d}\xi \\
&= c_n\sqrt{1 - h^2}Y_{lm}(\omega)\left[\int_0^\infty\rho^{n-2}  f_{lm}(\rho) \mathrm{d}\rho \right] \int_{-1}^1 \delta(h - t) C^{\lambda}_l(t) (1 - t^2)^{\lambda - 1/2}\mathrm{d}t \\
& = c_n(1 - h^2)^{(n-2)/2} Y_{lm}(\omega) C^{\lambda}_l(h)\left[\int_0^\infty\rho^{n-2}  f_{lm}(\rho) \mathrm{d}\rho \right],
\end{split}
\end{equation}
where the second step follows from the Funk-Hecke theorem. The result follows.
\end{proof}

Using Theorem \ref{thm_1}, we can relate the cone transform $R$ to the weighted ray transform
\begin{equation} \label{eq:WeightedRayDef}
D_kf(\vv, \omega) = \int_0^\infty \rho^{k}  f(\vv + \rho \omega) \mathrm{d}\rho.
\end{equation}
This relationship is given in the next corollary.

\begin{corollary}\label{corr_1}
Let $h_0 \in (-1,1)$ be transcendental. Then,
\begin{equation}
\label{exp_1}
D_{n-2}f(\vv, \omega) = \frac{(1-h_0^2)^{(2-n)/2}}{c_n} \sum_{lm} \frac{Rf_{lm}(\vv, h_0)}{C^{\lambda}_l(h_0)} Y_{lm}(\omega).
\end{equation}
\end{corollary}

\begin{proof}
When $\lambda$ is rational, the Gegenbaur polynomials are rational polynomials, and hence $C^{\lambda}_l(h_0) \neq 0$ for any $\lambda = (n-2)/2$ and $l$. Thus,
\begin{equation}
\begin{split}
D_{n-2}f(\vv, \omega) &= \sum_{lm}\int_0^\infty \rho^{n-2} f_{lm}(\rho) Y_{lm}(\omega) \mathrm{d}\rho \\
&= \frac{(1-h_0^2)^{(2-n)/2}}{c_n} \sum_{lm} \frac{Rf_{lm}(\vv, h_0)}{C^{\lambda}_l(h_0)} Y_{lm}(\omega),
\end{split}
\end{equation}
which follows from Theorem \ref{thm_1}.
\end{proof}
\begin{remark}
Corollary \ref{corr_1} has similarities to results from the literature \cite{terzioglu2020exact,moon2025inversion,kuchment2017inversion}, where cone integral data is converted to weighted divergent beam integrals. In those works, the authors use all $h \in (-1,1)$ to recover the divergent beam transform. We need only a single transcendental $h_0 \in (-1,1)$, which is an important advantage of Corollary \ref{corr_1}. However, our result relies on infinite series, whereas the results of the literature are closed form and take integrals over $h$. Thus, our method is likely less stable than previous methods as it uses less data.
\end{remark}

When $n = 2$ and $\s$ has certain properties, the next theorem shows how we can determine $D_{0}f$ from $Rf$ via a simple, stable algorithm. 

\begin{theorem}\label{simple_algo}
Let $n=2$ and $\s$ be a convex curve where $B$ is such that no line tangent to $\s$ intersects $B$. Then, for any fixed $\alpha \in (0,\pi/2)$,
\begin{equation}
D_{0}f(\vv,\omega) = \sum_{i=1}^{k(\omega, \alpha)} (-1)^{i-1} Rf(\vv,\omega + (2i-1)\alpha, \cos\alpha),
\end{equation}
where $k \geq 1$ is an integer and we are treating $\omega \in [0,2\pi]$ as an angle.
\end{theorem}
\begin{proof}
Without loss of generality, let $\vv = 0$, $T_\vv\s = \{x_2 = 0\}$ and $B \subset \{x_2 > 0\}$. Then, when $\omega \in [\pi,2\pi]$, $D_{0}f = 0$, so we consider the case when $\omega \in [0,\pi]$. Let
$$\gamma = \omega +2k \alpha \iff k = \varphi(\gamma) = \frac{\gamma - \omega}{2\alpha}.$$
Then, $\varphi([\pi,2\pi]) = [k_1,k_2]$, where $k_1 = \frac{\pi - \omega}{2\alpha}\geq 0$ and $k_2 - k_1 = \frac{\pi}{2\alpha} > 1$ since $\alpha < \pi/2$. Thus, there exists an integer $k\geq 1$ with $k \in [k_1, k_2]$. Let $k = k(\omega, \alpha)$ be one of these integers. Then,
\begin{equation}\label{algo}
\begin{split}
\sum_{i=1}^{k(\omega, \alpha)} (-1)^{i-1} Rf(\vv,\omega + (2i-1)\alpha, \cos\alpha) &= D_{0}f(\vv,\omega) + (-1)^{k-1}D_{0}f(\vv,\gamma)\\
 &= D_{0}f(\vv,\omega).
\end{split}
\end{equation}
The last step follows since the ray $L = \{\vv + t(\cos\gamma,\sin\gamma) : t\geq 0\} \subset \{x_2 \leq 0\}$, but the support of $f$ is contained in $\{x_2 > 0\}$ and so $D_{0}f(\vv,\gamma) = 0$.
\end{proof}
For example, in the above theorem, if we set $\alpha = \pi/4$, then $k = k(\omega,\alpha) = 2$ is constant and $\gamma = \omega + \pi$.

\begin{remark}\label{remark_noise}
    Overall, we can expect the above algorithm to become more unstable as $k$ increases. For an example of how this works, let $g_i = Rf_i + \eta_i$ be samples of V-line integral data with added random noise $\eta_i$, where the $\eta_i \sim \mathcal{N}(0,\sigma)$ are i.i.d. Gaussian random variables with mean zero and standard deviation $\sigma$. This noise model is most appropriate when the photon counts are sufficiently large, where $\sigma$ represents the noise level. If we set $2\alpha = \pi/k'$, then $k = k'$ in \eqref{algo}. In this case, when we apply \eqref{algo} to $g$, this yields samples of the form $g'_i = D_{0}f_i + \eta'_i$, where the $\eta'_i \sim \mathcal{N}(0,\sqrt{k'}\sigma)$ are i.i.d. Gaussian random variables. That is, when we convert samples of $Rf$ to $D_{0}f$ with added Gaussian noise in this way, the noise level is essentially increased by a factor of $\sqrt{k'}$. Of course, when $k$ is not constant, this calculation would likely become more complex, but this gives one an idea of how the noise is affected by the algorithm in \eqref{algo}.
\end{remark}

Now, using Corollary \ref{corr_1}, we can equivalently consider the transform
\begin{equation}\label{trans_data}
\mathcal{D}_k(f,\mu)(\vv, \omega) = \int_{0}^\infty \rho^{k} f_\vv(\rho,\omega) e^{ -\int_{0}^\rho\mu_\vv(r,\omega)\mathrm{d}r }  \mathrm{d}\rho,
\end{equation}
where $k = n-2$, $\vx - \vv = \rho\Theta$, and $f_\vv(\rho,\Theta) = f(\vv + \rho\Theta)$ (similarly for $\mu_\vv$). Similar to how $\mathcal{R}$ relates to $R$, we let $D_kf = \mathcal{D}_k(f,0)$, i.e., $D_k$ denotes the linear formulation and $\mathcal{D}_k$ non-linear.

The transform defined in \eqref{trans_data} is very similar to the transform studied in the identification problem in SPECT imaging \cite{holman2020spect,SeanHolman,natterer1981identification,solmon1995identification,stefanov2014identification}, except we have an additional weight $\rho^{n-2}$. In \cite{SeanHolman}, the authors provide conditions for $f$ and $\mu$ to be unique given $\mathcal{D}_0(f,\mu)$. As the conditions of \cite{SeanHolman} are not satisfied in our case, we use a different approach based on the theory of conormal distributions. We first focus on the recovery of the boundary of $\Omega$, which is detailed in the next section.

\subsection{Determining the boundary of $\Omega$} \label{sec:bdet}
In this section, we explain how under certain geometric conditions we can recover the singular support of $f$ (the boundary of $\Omega$) using the transformed data $\mathcal{D}_{n-2}(f,\mu)$. First, we have some definitions.
\begin{definition} \label{def:Strictconv}
Let $M$ be an $(n-1)$-dimension smooth manifold embedded in $\mathbb{R}^n$ and let $\vx \in M$. 
We say that $M$ is strictly convex near $\vx$ if the second fundamental (see \cite{lee2018introduction}) of $M$ is definite (positive or negative) at $\vx$. 
$M$ is said to be strictly convex almost everywhere if for almost all $\vx \in M$, $M$ is locally strictly convex near $\vx$. 
\end{definition}

For example, one could draw a 1D smooth curve in $\mathbb{R}^2$ that is not globally strictly convex, but strictly convex almost everywhere away from a finite set of inflection points.

\begin{definition}[convex hull] The \emph{convex hull} $\Omega_c$ of $\Omega$ is the smallest convex set that contains $\Omega$.

\end{definition}

We now state some of our key geometric assumptions needed for this section.

\begin{assumption}{(Geometric assumptions)}\label{geo_ass}
Let $\Xi = \partial\Omega \cup (\cup_{i=1}^m \partial\Omega_i)$.
\begin{enumerate}
\item\label{conv} $\Xi$ is strictly convex almost everywhere.
\item\label{crit} $\text{supp}(\mu)  \cap \Omega_c = \emptyset$ (i.e., there is not any additional attenuating material in the non-convex ``grooves" of the source)
\item\label{no triple}  There is no line tangent to $\Xi$ at three or more points.
\item\label{all lines} Every line through $B$ intersects $\s$, i.e., $\s$ sufficiently ``wraps around" $B$.
\item \label{no_nonconvex} Any line that is tangent to $\Xi$ at a point where $\Xi$ is not strictly convex, is not tangent to $\Xi$ at any other point.
\end{enumerate}
\end{assumption}

While the Assumptions \ref{geo_ass} do exclude some cases, notably when part of $\Xi$ is flat, parts \eqref{conv}, \eqref{no triple}, and \eqref{no_nonconvex} are generic. We now discuss consequences of these assumptions. First, we have the following proposition.

\begin{proposition}\label{prop_hull}
$\mathcal{D}_{n-2}(f,\mu)$ determines $\Omega_c$.
\end{proposition}
\begin{proof}
Let $\Gamma_0 = \{(\vv,\omega) : \mathcal{D}_{n-2}(f,\mu)(\vv,\omega) = 0\}$. Then, because of Assumption \ref{geo_ass}(\ref{all lines}),
\begin{equation}
\label{equ_hull}
\Omega_c = \overline{\mathbb{R}^n \backslash \paren{ \cup_{(\vv,\omega) \in \Gamma_0} L_{\vv,\omega} }},
\end{equation}
where $L_{\vv,\omega} = \{\vv +t\omega : t\in \mathbb{R}\}$.
\end{proof}

This is to say that we can ascertain a rough location and shape of the source simply by considering the zeros in the data.


As we will now describe, Assumption \ref{geo_ass}(\ref{crit}) allows us to rewrite $\mathcal{D}_{n-2}(f,\mu)$ in a separated form. Let
$$\Lambda = \cup_{\vx \in \Omega_c} \cup_{t\in [0,1]} \{\vv + t(\vx - \vv)\},$$
and, for some $\epsilon > 0$, let
\[
\Lambda_{\epsilon} = \{ \vx \in \mathbb{R}^n \ : \  \exists \vy \in \Lambda\ \text{s.t.}\ |\vx - \vy| < \epsilon\}
\]
be a slightly larger set. Suppose that $\varphi_{\vv} \in C_{c}^\infty(\mathbb{R}^3)$ is equal to $1$ on $\Lambda$ and has support contained in $\Lambda_{\epsilon}$. Then we define
\[
\tilde{\mu}_\vv(\rho,\omega) = \varphi_\vv(\rho,\omega) \mu_\vv(\rho,\omega)
\]
and, by Assumption \ref{geo_ass}(\ref{crit}),
\begin{equation}\label{prod_data}
\begin{split}
\mathcal{D}_{n-2}(f,\mu)(\vv,\omega) &= e^{ -\int_0^\infty \tilde{\mu}_\vv(\rho,\omega)\mathrm{d}\rho }  \int_0^\infty \rho^{n-2} f_\vv(\rho,\omega) \mathrm{d}\rho \\
&= e^{ -\int_0^\infty \tilde{\mu}_\vv(\rho,\omega)\mathrm{d}\rho }  D_{n-2}f(\vv,\omega).
\end{split}
\end{equation}
Note that $\WF(\mu_\vv) = \WF(\tilde{\mu}_\vv)$ on $T^*\Lambda_\epsilon$. 


Since the supports of $f$ and $\mu$ are assumed to be disjoint from $\s$, we can write $D_{n-2}f = X_{w_2}f$, where $X_{w_2}f(\vv,\omega) = \int_{\mathbb{R}}w_2(\vv,\omega,t) f(\vv + t\omega) \mathrm{d}t$ is a weighted X-ray transform with $w_2 \geq 0$ a smooth weight which is zero on the half-space $\{(\vx - \vv)\cdot\omega < 0\}$. Similarly, we can write $\int_0^\infty \tilde{\mu}_\vv(\rho,\omega)\mathrm{d}\rho = X_{w_1}\mu(\vv,\omega)$ for an appropriate smooth weight $w_1$ using the construction above. 
Then, \eqref{prod_data} becomes
\begin{equation}\label{prod_data_1}
\mathcal{D}_{n-2}(f,\mu)(\vv,\omega) = e^{ -X_{w_1}\mu(\vv,\omega) }  X_{w_2}f(\vv,\omega),
\end{equation}
where the $w_i \geq 0$ are smooth weights and $w_2 > 0$ on the region of interest, $B$. We next analyze the product in \eqref{prod_data_1} to classify $\mathcal{D}_{n-2}(f,\mu)$ in a $I^{p,l}$ space as introduced in Definition \ref{pair_lag}.

Let $\Gamma = \Gamma_f \cup \Gamma_\mu = \text{ssupp}(X\chi_{\Xi}) = \text{ssupp}(X(f+\mu))$ where $Xf = X_1 f$ is the unweighted X-ray transform of $f$, $\Gamma_f = \text{ssupp}(Xf)$, and $\Gamma_\mu = \text{ssupp}(X\mu)$. Equivalently, $\Gamma$ represents the set of all lines tangent to $\Xi$. Let $L_{\vv,\omega} = \{\vv + t\omega : t\in \mathbb{R}\}$ be a ray. Then, we consider two scenarios:
\begin{itemize}
\item Case 1 - $L_{\vv,\omega}$ is tangent to $\Xi$ at exactly one strictly convex point.
\item Case 2 - $L_{\vv,\omega}$ is tangent to $\Xi$ at exactly two strictly convex points.
\end{itemize}
Note, we do not need to consider triple tangencies by Assumption \ref{geo_ass}(\ref{no triple}). We first consider case 1 and show, in that case, that $\mathcal{D}_{n-2}(f,\mu)$ is locally a conormal distribution (see Definition \ref{def_conormal}). For the following, we will use the notation  $\Sigma = \cup_{i=1}^m \partial \Omega_i$.\\
\\
\textbf{Case 1.} Let $(\vx_0,\xi_0) \in N^*\Sigma$, where $\Sigma$ is locally strictly convex near $\vx_0$. Then, we can parameterize rays near $\vx_0$ using coordinates $\vy = (s,\vx,\theta)$ as follows. Let $L_\vy = \{\vx + s\xi(\vx) + t\theta : t\in \mathbb{R}\}$, where $s\in \mathbb{R}$, $(\vx,\xi)$ is on a local neighborhood of $(\vx_0,\xi_0)$ on $\WF(\mu)$, and $\theta \in \xi^{\perp} \cap S^{n-1}$, i.e., $\theta \in S^{n-1}$ is orthogonal to $\xi$. Here, $\vx \in \Sigma$ near $\vx_0$ and $\xi(\vx)$ is the corresponding normal to $\Sigma$ at $\vx$, and we assume $\xi$ points towards the convex side of $\Sigma$. Then, we can reformulate the X-ray transform in these coordinates
\begin{equation}
X_w\mu(\vy) = \int_{\mathbb{R}} w(\vy) \mu(\vx + s\xi(\vx) + t\theta) \mathrm{d}t,
\end{equation}
locally near $\vy_0 = (0,\vx_0,\theta_0)$, where $L_{\vy_0}$ is tangent to $\Sigma$ at $\vx_0$. By Assumption \ref{geo_ass}(\ref{all lines}), we can thus reformulate \eqref{prod_data_1} in these coordinates near $\vy_0$,
\begin{equation}\label{prod_data_2}
\mathcal{D}_{n-2}(f,\mu)(\vy) = e^{ -X_{w_1}\mu(\vy) }  X_{w_2}f(\vy),
\end{equation}
where the $w_i \geq 0$ are smooth weights and $w_2 > 0$ on $B$.

Now we have the lemma which describes the singularities of the X-ray transform in these coordinates.
\begin{lemma}
Let $L_{\vy_0}$ be tangent to $\Sigma$ at only one point $\vx_0$, where $\Sigma$ is locally strictly convex near $\vx_0$. Then, $X_w\mu(\vy)$ takes the form
\begin{equation}\label{form_sr}
X_{w}\mu(\vy) = g_0(\vy) + g_1(\vy)s^{1/2}_+,
\end{equation}
near $\vy_0$, where $\vx \in \Sigma$ is near $\vx_0$, the $g_i \geq 0$ are smooth.
\end{lemma}

\begin{proof}
The set $\Gamma$ of lines tangent to $\Sigma$ near $\vx_0$ is defined by $\{s = 0\}$ locally.
Now, letting $\vx \in \Sigma$ near $\vx_0$ and $\theta \in \xi^{\perp}\cap S^{n-1}$, we define the part of 2-D plane 
$$\mathcal{L} = \cup_{s\in \mathbb{R}} \{\vx + s\xi(\vx) + t\theta : t\in \mathbb{R}\},$$
which can be paramaterized in terms of $(s,t)$. Then, for $\epsilon > 0$ small enough, $\gamma = B_\vx(\epsilon) \cap \Sigma \cap \mathcal{L}$ is a smooth strictly convex 1-D curve embedded in $\mathcal{L}$. This is true because $\Sigma$ is locally strictly convex near $\vx_0$. Let $\tilde{\mu}(s,t) = \mu(\vx + s \xi(\vx) + t \theta) : \mathbb{R}^2\to \mathbb{R}$ and $\tilde{w}(s,t) = w(\vx + s \xi(\vx) + t \theta) : \mathbb{R}^2\to \mathbb{R}$. Then, $X_w\mu(s,\vx,\theta) = \int_\mathbb{R} \tilde{w}(s,t) \tilde{\mu}(s,t) \mathrm{d}t$ is simply a weighted classical 1-D Radon projection of $\tilde{\mu}$. Since by construction, $\tilde{\mu}$ has a jump singularity along $\gamma$, and $\gamma$ is strictly convex near $\vx$, the Radon projection adopts a square root type singularity \cite{katsevich2026analysis}, thus leading to \eqref{form_sr}.
\end{proof}

Using the above lemma, we have
\begin{equation}
X_{w_1}\mu(\vy) = g_0(\vy) + g_1(\vy)s^{1/2}_+,
\end{equation}
near $\vy_0$, where the $g_i \geq 0$ are smooth. We now explain how the exponential function affects such a distribution following similar ideas to \cite{katsevich2026analysis}. Using the Taylor series of $e^{-x}$ about $x = 0$, we have
\begin{equation}\label{taylor}
\begin{split}
 e^{-(g_0 + g_1s^{1/2}_+)} &= e^{-g_0}\paren{ 1 + \sum_{k=1}^\infty (-1)^k \frac{g_1^k s^{k/2}_+}{k!} }\\
&= e^{-g_0}\paren{ 1 +  h_o s^{1/2}_+ + h_e s_+  },
\end{split}
\end{equation}
where
$$h_o = \sum_{k\geq 1\ \text{odd}}^\infty (-1)^k \frac{g_1^k s^{(k-1)/2}}{k!},$$
and
$$h_e = \sum_{k\geq 2 \ \text{even}}^\infty (-1)^k \frac{g_1^k s^{(k-2)/2}}{k!},$$
are both smooth. Thus, we have
\begin{equation}
e^{-X_{w_1}\mu(\vy)} = h_0(\vy) + h_1(\vy)s^{1/2}_+ + h_2(\vy) s_+,
\end{equation}
near $\vy_0$, where the $h_i$ are smooth and $h_0 = e^{-g_0} > 0$. Note that $2(n-1)$ is the dimension of the set of lines in $\mathbb{R}^n$. 
Thus, using Proposition \ref{prop_sqrt}, we have
\begin{equation}
e^{-X_{w_1}\mu} \in I^{- \frac{(n+1)}{2}}(N^*\Gamma) = I^{- \frac{3}{2}}(\Gamma)
\end{equation}
near $\vy_0$. Since $L_{\vy_0}$ is tangent to $\Xi$ only at $\vx_0 \in \Sigma$ and $\Sigma \cap \partial \Omega = \emptyset$, $X_{w_2}f$ is clearly smooth near $\vy_0$ and thus
\begin{equation}
\mathcal{D}_{n-2}(f,\mu) \in  I^{- \frac{(n+1)}{2}}(N^*\Gamma)
\end{equation}
near $\vy_0$. Note, the calculations above are analogous (and slightly easier) if $(\vx_0,\xi_0) \in N^* (\partial\Omega)$ instead of $N^* \Sigma$, so that case is left to the reader. However, the result is the same. Using Definition \ref{def_conormal}, it follows that these singularities of the data (i.e., near such $\vy_0$) are in $H^{\alpha}$ microlocally for $\alpha < 1$.\\
\\
\textbf{Case 2.} Now, we consider double tangent lines and show that $\mathcal{D}_{n-2}(f,\mu)$ is locally in an $I^{p,l}$ class. The ray can either be double tangent to $\partial \Omega$ or $\Sigma$, or single tangent to $\partial \Omega$ and $\Sigma$. Let $L_{\vy_0}$ be tangent to $\Xi$ at two distinct points $\vx_1, \vx_2 \in \Sigma$, where $\Sigma$ is locally strictly convex near $\vx_1,\vx_2$. Note, we can parameterize rays locally near $\vy_0$ in this case using local neighborhoods of either $\vx_i$ on $\Sigma$ as in the single tangent case. The other possibilities (e.g., when $L_{\vy_0}$ is single tangent to $\partial \Omega$ and $\Sigma$) are again analagous. Let $\varphi_1, \varphi_2$ be smooth cutoffs centered on $\vx_1,\vx_2$. Then, locally near $\vy_0$, $\Gamma = \Gamma_1 \cup \Gamma_2$ is the union of two intersecting smooth manifolds, where $\Gamma_i = \ssupp(X(\varphi_i \mu))$ near $\vy_0$.

\begin{proposition}\label{prop_transverse}
The $\Gamma_i$ intersect transversally at $\vy_0$. 
\end{proposition}
\begin{proof}
Let $\xi_i$ be the corresponding covector to $\vx_i$ on $\Sigma$, so that $(\vx_i,\xi_i) \in N^*\Sigma$. Let $\mathcal{C}_X$ denote the canonical relation of $X$. Let us assume the $\Gamma_i$ do not intersect tranversally at $\vy_0$. Then, the corresponding covectors to $\vy_0$ on $\Gamma_i$ are both parallel to some $\eta$. This means that $\Cc_X(\vx_1,\xi_1) = \Cc_X(\vx_2,\xi_2) = (\vy_0,\eta)$, but this is not possible since $\vx_1 \neq \vx_2$ and $X$ satisfies the Bolker condition (\cite[Proposition 3.1]{monard2015geodesic}). Thus, we have a contradiction and the $\Gamma_i$ intersect transversally at $\vy_0$.
\end{proof}

Proposition \ref{prop_transverse} was also proven in a more general context in \cite[Lemma 4.2]{chihara2026geodesic}. Let $\Cc_X(\vx_i,\xi_i) = (\vy_0,\eta_i)$ for $i = 1,2$, where $\eta_1$ is not parallel to $\eta_2$. Then, in this case,
\begin{equation}\label{prod_d}
\mathcal{D}_{n-2}(f,\mu)(\vy) = \varphi(\vy) e^{-X_{w_1}(\varphi_1\mu)} \cdot e^{-X_{w_1}(\varphi_2\mu)} = \varphi(\vy) u_1(\vy)u_2(\vy),
\end{equation}
near $\vy_0$, where $\varphi$ is smooth and $u_i \in  I^{- \frac{3}{2}}(\Gamma_i)$, noting that $X_{w_2}f$ is smooth near $\vy_0$. Since the $\Gamma_i$ intersect transversally at $\vy_0$, we have from {\cite[Theorem 8.2.10]{hormanderI}
$$\WF_{\vy_0}(\mathcal{D}_{n-2}(f,\mu)) \subset \text{span}(\eta_1,\eta_2).$$
That is, taking the product of the $u_i$ as above can create new singularities in the data at $\vy_0$ in directions which are not parallel to either $\eta_i$, i.e., singularities which are not part of the linear X-ray transform. 

We now have the following theorem.
\begin{theorem}\label{double}
Let $L_{\vy_0}$ be tangent to $N^* \Xi$ at precisely two points $(\vx_i,\xi_i)$, $i = 1,2$, where $\Xi$ is locally strictly convex near each $\vx_i$.
Then
\begin{equation}
\mathcal{D}_{n-2}(f,\mu) \in I^{- \frac{(n+1)}{2},-1}(N^*(\Gamma_1 \cap \Gamma_2),N^*\Gamma_1) + I^{- \frac{(n+1)}{2},-1}(N^*(\Gamma_1 \cap \Gamma_2),N^*\Gamma_2)
\end{equation}
locally near $\vy_0$.
\end{theorem}
\begin{proof}
By \eqref{prod_d}, we have, locally near $\vy_0$,
\begin{equation}
\mathcal{D}_{n-2}(f,\mu)(\vy) = \varphi (u_1 \cdot u_2),
\end{equation}
where $\varphi$ is smooth and $(u_1 \cdot u_2)\in I^{- \frac{3}{2}}(\Gamma_1) \cdot I^{- \frac{3}{2}}(\Gamma_2)$. Since the $\Gamma_i$ intersect transversally by Proposition \ref{prop_transverse}, we can apply Lemma \ref{lem_green} to yield
\begin{equation}
\mathcal{D}_{n-2}(f,\mu)\in  I^{- \frac{3}{2},- \frac{3}{2}}(\Gamma_1,\Gamma_1 \cap \Gamma_2) + I^{- \frac{3}{2},- \frac{3}{2}}(\Gamma_2, \Gamma_1 \cap \Gamma_2)
\end{equation}
locally near $\vy_0$. Now, the $\Gamma_i$ are smooth codimension 1 submanifolds of $\mathbb{R}^k$, where $k = 2(n-1)$. Since they intersect transversally, their intersection is a codimension 2 smooth submanifold of $\mathbb{R}^k$. Thus, using the notation of Definition \ref{pair_lag}, we can calculate $d_1 = d_2 = 1$, and using \eqref{pair_lag_e} yields the result.
\end{proof}

Thus, it follows from the above theorem that 
\begin{equation}
\mathcal{D}_{n-2}(f,\mu) \in I^{- \frac{(n+1)}{2}-1}(N^*(\Gamma_1 \cap \Gamma_2)\backslash (N^*\Gamma_1 \cup N^* \Gamma_2) ),
\end{equation}
and $\mathcal{D}_{n-2}(f,\mu)$ is in $H^{\alpha}$ microlocally near $(\vy_0,\eta)$ for any $\eta$ not parallel to $\eta_1, \eta_2$, and for $\alpha < 2$, using Definition \ref{def_conormal}. Putting this all together, we see there are two distinct types of singularities in the data: those which correspond to the linear X-ray transforms of $\mu$ and $f$, which blow up in $H^1$, and those which occur at double tangent rays to $\Xi$, which are in $H^{\alpha}$ for $\alpha < 2$ (i.e., they do not blow up in $H^1$). The singularities which correspond to the linear transform (which we want to tease out) are thus stronger than those created by the non-linearities.

Singularities in the data will also occur at lines tangent to $\Xi$ at points where $\Xi$ is not strictly convex. These points will correspond to higher order data singularities in general, although we do not analyze them in more detail. In the generic situation described by Assumption \ref{geo_ass}, these higher order singularities will not affect our ability to recover $\partial \Omega$ as we will see later in Corollary \ref{corr_bound}.

\subsection{Singularities of the reconstruction}\label{sing_recon} Now that we have described the singularities of the data, we analyze how these singularities translate to the image domain.

Let $g = \mathcal{D}_{n-2}(f,\mu)$. Then, we aim to recover the singularities of $f$ via the reconstruction
\begin{equation}
f_r = X^{-1}g,
\end{equation}
where $X^{-1}$ is the (unweighted) inverse X-ray transform \cite[Theorem 2.1]{natterer}. Let $\rho : T^*B \to B$ be the natural projection. Let $\mathcal{L}_a$ denote the union of all double tangent rays in the data set. Then, by the above analyses, we have
\begin{equation}
\ssupp(X^{-1}g) \subset \Xi \cup \mathcal{L}_a,
\end{equation}
where the $\vx \in \mathcal{L}_a\backslash \Xi$ satisfy $\vx = \rho\paren{ C_X^t(\vy_0,\eta) }$, where $\vy_0$ corresponds to a double tangent ray, and
$$\eta \in \text{span}(\eta_1,\eta_2) \backslash (\text{span}(\eta_1) \cup \text{span}(\eta_2)),$$
where the $\eta_i$ are as described in Theorem \ref{double}. It is noted also that, since $w_2 > 0$ on $B$, $\text{ssupp}(Xf) \subset \text{ssupp}(g)$. Combining this with the above analysis shows that almost all elements of $\partial \Omega$ are included in $\ssupp(X^{-1}g)$, potentially excluding the parts of $\partial \Omega$ which are not locally strictly convex, which by Assumption \ref{geo_ass}(\ref{conv}) is a set of measure zero on $\partial \Omega$. $\ssupp(X^{-1}g)$ can also include elements of $\Sigma$, which we do not desire at this stage as we are interested in recovering $\partial \Omega$. To remedy this, let $\psi$ be a smooth cutoff which is equal to 1 on $\Omega_c$ (the convex hull of $\Omega$) and zero on $\text{supp}(\mu)$. This is possible since we have assumed that $\Omega_c\cap \text{supp}(\mu) = \emptyset$. Note, $\Omega_c$ can be determined using Proposition \ref{prop_hull}. Then, we have 
\begin{equation}
\ssupp(\psi X^{-1}g) \subset \partial \Omega \cup \mathcal{L}_a,
\end{equation}
and we can remove the singularities of $\mu$ in the reconstruction. It is also true that almost all elements of $\partial \Omega$ are contained in $\ssupp(\psi X^{-1}g)$, so we do not lose singularities of $f$ by multiplying by $\psi$.

It is well known that $X^{-1} : H_{comp}^{\alpha} \to H^{\alpha - 1/2}_{loc}$ is a continuous map\footnote{This can be argued from the facts that $X^* X$ is a pseudodifferential operator of order $-1$, and $X^{-1}$ can be written as composition of $X^*$ and a pseudodifferential operator of order $1$.}. Thus, if we apply $X^{-1}$ microlocally to singularities in the data created by the non-linearities, we see that the corresponding singularities in the reconstruction are in $H^{\alpha}$ for $\alpha < 3/2$, and this follows from Theorem \ref{double} and the following discussions. The singularities in the reconstruction which correspond to $\partial \Omega$ are stronger and contained in $H^{\alpha}$ for $\alpha < 1/2$. More precisely, if we let $f_r = \psi X^{-1}g$ and $\psi'$ be a smooth cutoff centered on $\vx \in \mathcal{L}_a\backslash \partial \Omega$, $\psi' f_r \in H^{\alpha}(\mathbb{R}^n)$ for $\alpha < 3/2$. In contrast, if $\psi'$ is centered on $\partial \Omega$, the Sobolev order is reduced by 1. This leads to the recovery of $\partial \Omega$.
\begin{corollary}\label{corr_bound}
Let $\psi$ be a smooth cutoff on $\Omega_c$ as described above and $g = \mathcal{D}_{n-2}(f,g)$. Then 
\begin{equation}
\overline{\ssupp^{1/2}(\psi X^{-1}g)} = \partial \Omega.
\end{equation}
\end{corollary}

\section{Simultaneous recovery of $f$ and $\mu$}\label{sec:sim}
In this section, using our already established microlocal theory,  we show, under reasonable assumptions on $\text{supp}(\mu)$ and $\text{supp}(f)$, that $f$ and $\mu$ are uniquely determined by $\mathcal{D}_{n-2}(f,\mu)$.

First, we have the theorem which shows how we can determine $f$ based on the theory of the last section.
\begin{theorem}\label{rec_f}
Let Assumption \ref{geo_ass} hold and let $f = I_0 \chi_\Omega$ where $I_0 > 0$ is a constant intensity and $\Omega \subset B$ is a compact domain with smooth boundary which is strictly convex almost everywhere. Let us further assume that there exists a $(\vv,\omega)$ such that $X_{w_2}f(\vv,\omega) > 0$ and $X_{w_1}\mu(\vv,\omega) = 0$. Then $\mathcal{D}_{n-2}(f,\mu)$ determines $f$ uniquely.
\end{theorem}

\begin{remark}
Given the construction of $w_1$, the final condition of the theorem means there exists a half-ray emanating from $\partial \Omega$ (i.e., the boundary of $f$) which reaches $\s$ without hitting $\mu$ (i.e., without attenuating). While this may seem a specific assumption, it is also required in our upcoming theorems (e.g., see Theorem \ref{rec_mu}), and will be needed throughout this section to recover $\mu$. 
\end{remark}
\begin{proof}
As was proven in the previous section, under Assumption \ref{geo_ass} and Corollary \ref{corr_bound}, the boundary of $\Omega$ ($\partial \Omega$) can be determined from $g = \mathcal{D}_{n-2}(f,\mu)$ by applying FBP type reconstruction and extracting the stronger singularities from the resulting image. The stronger singularities contained in the convex hull $\Omega_c$, which can be determined from the data data as shown in Proposition \ref{prop_hull}, must correspond to the boundary of $\Omega$. We now argue as in \cite{Webber2026Microlocal} that $\partial \Omega$ determines $\Omega$. Indeed, take any $\vx \in \mathbb{R}^n\setminus \partial \Omega$ and let $L_\vx$ be an infinite ray with $\vx$ as the origin which is never tangent to $\partial \Omega$. Let $k$ be the number of points in $\partial \Omega \cap L_\vx$.  Since $\Omega$ is compact, points sufficiently far along this ray will be outside of $\Omega$ and from this we can conclude that if $k$ is odd $x \in \Omega$ and if $k$ is even then $\vx \notin \Omega$. Thus, $\Omega$ is determined.

Now, by \eqref{prod_data_2}, we have
\begin{equation}
\mathcal{D}_{n-2}(f,\mu)(\vv,\omega) = e^{ -X_{w_1}\mu(\vv,\omega) }  X_{w_2}f(\vv,\omega),
\end{equation}
and, since $X_{w_2}$ is linear, $X_{w_2}f = I_0X_{w_2}\chi_\Omega$. Rearranging the above gives
\begin{equation}
e^{ -X_{w_1}\mu(\vv,\omega) }I_0 = \frac{ \mathcal{D}_{n-2}(f,\mu)(\vv,\omega) }{X_{w_2}\chi_\Omega(\vv,\omega) },
\end{equation}
where the right-hand side is known. By assumption that there exists a $(\vv,\omega)$ with $X_{w_2}\chi_\Omega(\vv,\omega) > 0$ and $X_{w_1}\mu(\vv,\omega) = 0$, we can recover $I_0$ via
$$I_0 = \text{max}_{(\vv,\omega) \in \mathcal{Y}} \frac{ \mathcal{D}_{n-2}(f,\mu)(\vv,\omega) }{X_{w_2}\chi_\Omega(\vv,\omega) },$$
where $\mathcal{Y} = \{(\vv,\omega) : \mathcal{D}_{n-2}(f,\mu)(\vv,\omega) \neq 0\}$. This completes the proof.
\end{proof}

We now explain how we can determine $\mu$ uniquely in certain situations. Taking the negative log of \eqref{prod_data_2} yields
\begin{equation}\label{recov_mu}
-\log \mathcal{D}_{n-2}(f,\mu)(\vv,\omega) = X_{w_1}\mu(\vv,\omega) - \log X_{w_2}f(\vv,\omega),
\end{equation}
and thus, noting that $f$ is determined by the above theorem, we can determine $X_{w_1}\mu(\vv,\omega)$ for any $(\vv,\omega)$ such that $L_{\vv,\omega}$ intersects $\Omega$, i.e., when $X_{w_2}f(\vv,\omega) \neq 0$. For the next argument, we remind the reader that construction of the weight $w_1$ is given after Assumption \ref{geo_ass}. As $L_{\vv,\omega}$ emanates from $\s$ and passes through $\Omega$, $w_1 = 0$ on the half ray exiting $\Omega$ along $L_{\vv,\omega}$ in direction $\omega$. Given this construction and that $\text{supp}(\mu) \cap \Omega_c = \emptyset$, it is clear that $ X_{w_1}\mu$ determines $D_0\mu(\vv,\omega)$ (recall \eqref{eq:WeightedRayDef}) for $\vv$ in the interior of $\Omega$ and for $\omega$ such that $L_{\vv,\omega}$ intersects $\s$. That is, $\mathcal{D}(f,\mu)$ determines all divergent beam integrals of $\mu$ emanating from $\Omega$ and passing through $\s$. Intuitively, this means that, once we have determined the source location and intensity, we can use rays emitted from the source to recover the attenuation of the surrounding material. With this in mind, we now establish conditions so that $D_0\mu$ determines $\mu$. 

First, we expand upon a uniqueness theorem previously proven by Natterer \cite[Theorem 3.3, chapter 2]{natterer} who considered the nearly equivalent result for the case when $\mu \in C_c^\infty(\mathbb{R}^n)$. We are only really interested in the case when $\mu$ is piecewise constant, but establish the result in the more general situation $\mu \in \mathcal{E}'(\mathbb{R}^n)$, the space of compactly supported distributions.

Before presenting our uniqueness theorem, we need to consider how $D_k$ is extended to $\mu\in \mathcal{E}'(\mathbb{R}^n)$. 
For this theorem, we actually consider a restricted version of $D_k$ which we will denote $\widetilde{D}_k$. 
Suppose $A$ is a smooth curve in $\mathbb{R}^n$ parametrized by $I \ni s \mapsto \va(s) \in A$, where $I$ is an open interval, and $\Lambda \subset S^1$ is an open set. Also set
\[
\Lambda_c = \left \{ \vx \in \mathbb{R}^n \ : \ \frac{\vx}{|\vx|} \in \Lambda \right \}
\]
Then for $\mu \in C_c^\infty(\mathbb{R}^n)$ and $k \in \mathbb{Z}_{\geq 0}$ we define
\begin{equation}
    \widetilde{D}_k \mu(s,\vx) = \int_0^\infty t^{k} \mu(\va(s) + t \vx) \d t
\end{equation}
which continuously maps $C_c^\infty(\mathbb{R}^n)$ to $C^\infty(I \times \Lambda_c)$. Note that
\[
\widetilde{D}_k \mu(s,\vx) = D_k\mu(\va(s),\vx/|\vx|) |\vx|^{-(k+1)}
\]
and so $\widetilde{D}_k \mu$ contains the same information as $D_k \mu$ with vertices restricted to $A$ and directions restricted to $\Lambda$. 

If we take a test function $h \in C_c^\infty(I\times \Lambda_c)$, then
\begin{equation}\label{eq:Dkdef2}
\begin{split}
\langle \widetilde{D}_k \mu, h \rangle_{L^2(I \times \Lambda_c)} & = \int_{\Lambda_c} \int_I \int_0^\infty \rho^{k} \mu(\va(s) + t \vx) h(s,\vx) \ \mathrm{d}s \mathrm{d}t \mathrm{d}\vx \\
& = \int_{\mathbb{R}^n} \mu(\vy)\left ( \int_I \int_0^\infty t^{k-n} h\left ( s, \frac{\vy-\va(s)}{t} \right )\ \mathrm{d}t \mathrm{d}s\right ) \mathrm{d}\vy
\end{split}
\end{equation}
where we have changed variables in the second step and applied Fubini's theorem. From this calculation we see that formal adjoint $\widetilde{D}_k^*$ is given by
\[
\widetilde{D}_k^* h(y) =  \int_I \int_0^\infty t^{k-n} h\left ( s, \frac{\vy-\va(s)}{t} \right )\ \mathrm{d}t \mathrm{d}s
\]
and it is clear $\widetilde{D}_k^* : C_c^\infty(I \times \Lambda_c) \rightarrow C^\infty(\mathbb{R}^n)$ continuously. Based on this, we can extend $\widetilde{D}_k\mu$ to act on $\mu \in \mathcal{E}'(\mathbb{R}^n)$ by the formula
\[
\langle \widetilde{D}_k \mu, h \rangle = \langle \mu, \widetilde{D}_k^* h \rangle.
\]
The general theory says $\widetilde{D}_k :  \mathcal{E}'(\mathbb{R}^n) \rightarrow \mathcal{D}'(I \times \Lambda_c)$ continuously. Note that $\mathcal{D}'$ here is the space of distributions, not to be confused with the non-linear operator $\mathcal{D}$.

\begin{theorem}
\label{natt}
Let $\Lambda$ be an open set on $S^{n-1}$, and let $A$ be a continuously
differentiable curve in $\mathbb{R}^n$ parametrized by $I \ni s \mapsto \va(s) \in A$. Suppose that the operator $\widetilde{D}_0: \mathcal{E}'(\mathbb{R}^n)\rightarrow \mathcal{D}'(I \times \Lambda_c)$ is defined as above using $A$ and $\Lambda$. Let $U \subset \mathbb{R}^n$ be bounded, open and disjoint from $A$. Assume that for each $\omega \in \Lambda$
there is an open interval $I_\omega$ such that for all $s \in I_w$, the ray $\{\va(s) + t\omega \ :\ t\geq 0\}$ misses $U$. If $\mu \in \mathcal{E}'(U)$ and
$\widetilde{D}_0\mu=0$, then $\mu = 0$ when restricted to the ``measured region"
$\{\vv+t\omega : \vv\in A, \omega \in \Lambda\}$.
\end{theorem}
\begin{proof}

Suppose $\mu$ satisfies the hypotheses of the theorem. We will prove by induction that $\widetilde{D}_k \mu = 0$ for all $k$, which we already know is true for $k= 0$. So, suppose $\widetilde{D}_{k-1} \mu = 0$ for some $k \geq 1$.

For $h(\cdot_s,\cdot_\vx) \in C_c^\infty(I \times S_c)$, using the identity
\[
\partial_s \left ( h \left ( s , \frac{\vy - \va(s)}{t} \right ) \right ) = \partial_{\widetilde{s}}|_{\widetilde{s} = s} h \left ( \widetilde{s} , \frac{\vy - \va(s)}{t} \right ) - \frac{\va'(s)^T}{t} \nabla_\vx h \left ( s , \frac{\vy - \va(s)}{t} \right )
\]
and a variation on integration-by-parts, we can calculate
\[
\begin{split}
\widetilde{D}_k^* [\partial_sh](\vy) & = \int_I \int_0^\infty t^{k-n} \partial_{\tilde{s}}|_{\tilde{s} = s} h\left ( \tilde{s}, \frac{\vy-\va(s)}{t} \right )\ \mathrm{d}t \mathrm{d}s\\
& =  \int_I \int_0^\infty t^{k-n-1} \va'(s)^T \nabla_\vx h\left ( s, \frac{\vy-\va(s)}{t} \right )\ \mathrm{d}t \mathrm{d}s \\
& = \widetilde{D}_{k-1}^* [(\va')^T \nabla_\vx h](\vy).
\end{split}
\]
Therefore,
\begin{equation}\label{eq:dDkfzero}
\langle \partial_s \widetilde{D}_k \mu, h \rangle = -\langle \mu,\widetilde{D}^*_k \partial_s h \rangle = -\langle \mu,\widetilde{D}^*_{k-1} (\va')^T\nabla_x h \rangle = -\langle \widetilde{D}_{k-1} \mu, (\va')^T\nabla_x h \rangle = 0.
\end{equation}
Now, by hypotheses we can find a covering of $\Lambda$ by open sets $\Lambda_i$ such that for each $i$ there is an open interval $I_i \subset I$ such that for all $\omega \in \Lambda_i$ and $s \in I_i$ the ray $\{\va(s) + t\omega \ :\ t\geq 0\}$ misses $U$. Let $\{\varphi_i\}$ be a partition of unity subordinate to the covering of $\Lambda$, and for $h \in C_c^\infty(I \times S_c)$ we have
\[
\langle \widetilde{D}_k\mu,h \rangle = \sum_{i} \langle \widetilde{D}_k\mu, \varphi_ih \rangle.
\]
We will prove each of the terms in the sum above is equal to zero. Now let $I_{s_0,\epsilon} = [s_0-\epsilon,s_0+\epsilon] \subset I_i$ and $\{s_j\}_{j=0}^N \subset I$ be a finite collection of points such that the intervals $I_{s_j,\epsilon} = [s_j - \epsilon,s_j+\epsilon]$ cover $I$. Let $\{\phi_j\}_{j=0}^N$ be a subordinate partition of unity on $I$. Then it is also sufficient to prove the terms
\[
\langle \widetilde{D}_k f,\varphi_i \phi_j h \rangle
\]
are all zero. For convenience, let us write \[
h_{ij} = \varphi_i \phi_j h
\]
and note that the support of $h_{ij}$ is contained in $I_{s_j,\epsilon} \times \Lambda_{i,c}$. Using the fundamental theorem of calculus, we have
\[
\begin{split}
\langle \widetilde{D}_k \mu, h_{ij} \rangle & = \left \langle \widetilde{D}_k \mu, -\int_0^{s_0-s_j} \partial_s h_{ij}(s-c,\vx) \ \mathrm{d}c + h_{ij}(s+s_j - s_0,\vx) \right \rangle\\
& = \int_{0}^{s_0-s_j} \langle \partial_s \widetilde{D}_k \mu, h_{ij}(s-c,\vx)\rangle \ \dd c + \langle \widetilde{D}_k \mu, h_{ij} (s+s_j-s_0,\vx) \rangle.
\end{split}
\]
The first term above is zero because of \eqref{eq:dDkfzero} while we can see the second is zero by considering \eqref{eq:Dkdef2} and noting that the support of $(s,\vx) \mapsto h_{ij}(s+s_j-s_0,\vx)$ is contained in $I_{s_0,\epsilon}\times \Lambda_{i,c} \subset I_i \times \Lambda_{i,c}$. This proves by induction that $\widetilde{D}_k \mu = 0$ for all $k \geq 0$.

By taking a partition of unity we can split up the measured region and so it is sufficient to prove
\[
\langle \mu, g \rangle = 0
\]
whenever the support of $g$ is contained in a small sliver in the measured region given by $\{ \va(s_0) + t \omega \ : \ \omega \in \Lambda_\epsilon\}$ where $\Lambda_\epsilon \subset \Lambda$ is sufficiently small. Approximating $g$ by polynomials in the radial variable centered at $a(s_0)$, which is possible by the Stone-Weierstrass theorem since the support of $\mu$ is compact, we can use the fact $\widetilde{D}_k 
\mu = 0$ for all $k$ to show that $\mu = 0$ when restricted to the measured region.
\end{proof}


We now prove injectivity results for $D_0$ in the context of our application. We remark that the theorems in this section can be extended to apply to $D_k$ for $k > 0$, although this is not necessary for the application we consider.

\begin{theorem}\label{rec_mu}
Let $B = \{|\vx| <1\}$ be the unit ball in $\mathbb{R}^2$ and suppose $\Omega \Subset B$ has smooth boundary denoted by $A = \partial \Omega$. Let there exist $\vv_0 \in A$ and some open $\Lambda \subset S^1$ such that $\{\vv_0 + t\omega : t\geq 0\}$ misses $U = \text{supp}(\mu)$ for any $\omega \in \Lambda$. Let us further assume that $\mu \in L^\infty(\mathbb{R}^2)$ with $\text{supp}(\mu) \subset B$ and that $\mu = 0$ when restricted to the interior of $A$. Then if $D_0\mu(\vv,\omega) = 0$ for any $\vv\in A$ and $\omega \in S^1$, $\mu = 0$.
\end{theorem}

\begin{proof}
Let $D_0\mu(\vv,\omega) = 0$ for any $\vv\in A$ and $\omega \in S^1$ and choose a circle $C \subset \Omega$ such that $\vv_0 \in C$. Then, by assumption that $\mu = 0$ on $\Omega$, $D_0\mu(\vv,\omega) = 0$ for any $\vv\in C$ and $\omega \in S^1$. For any $\vv$, define the cone $\{\vv + \Lambda\} = \{\vv + t\omega : \omega \in \Lambda, t\geq 0\}$. Then, using Theorem \ref{natt}, since $\vv_0 \in C$ and by assumption that $\{\vv_0 + t\omega : t\geq 0\}$ misses $\text{supp}(\mu)$ for any $\omega \in \Lambda$, $\mu = 0$ on 
$$\Delta = \cup_{\vv \in C} \{\vv + \Lambda\}.$$
Therefore, $\mu = 0$ on
$$\Delta_{\omega_0} = \cup_{\vv \in C} \{\vv + t\omega_0: t\geq 0\},$$
for any $\omega_0 \in \Lambda$ and $\text{supp}(\mu) \cap \Delta_{\omega_0} = \emptyset$.

Let $L_1 = \{\vx \cdot \omega^{\perp}_0 = s_1\}$ and $L_2 = \{\vx \cdot \omega^{\perp}_0 = s_2\}$ be the two lines tangent to $C$ which are parallel to $\omega_0$, and let $\vy_1$ and $\vy_2$ be their respective points of intersection with $C$. Let the half rays $\{\vy_1 + t\omega_0\}$ and $\{\vy_2 + t\omega_0\}$, $t\geq 0$, intersect $S^1$ at $\vz_1$ and $\vz_2$, respectively. Let $\omega_1 = (\vz_1 - \vy_2)/|\vz_1 - \vy_2|$. See figure \ref{fig_proof}. Then, $\text{supp}(\mu) \cap \{\vy_2 + \Lambda_{\omega_1}\} = \emptyset$, where $\Lambda_{\omega_1}$ is some small neighborhood of $\omega_1$ on $S^1$.  Note, by assumption, $\mu = 0$ outside of $S^1$. Thus, by Theorem \ref{natt}, $\mu = 0$ on 
$$\Delta_{\omega_1} = \cup_{\vv \in C} \{\vv + t\omega_1: t\geq 0\},$$
where 
$$\cup_{\vv \in C} \{\vv + t\omega : \omega \in \tilde{S^1}, t\geq 0 \} \subset (\Delta_{\omega_0} \cup \Delta_{\omega_1}),$$
and where $\tilde{S^1}$ is the shortest curve connecting $\omega_0$ and $\omega_1$ on $S^1$. Continuing this process, noting that the $\omega_k$ will eventually make a full rotation of $S^1$ after a finite number of steps (the number of iterations needed would depend on the size and location of $C$), we can see that $\mu = 0$ on 
$$\cup_{\vv \in C} \{\vv + t\omega : \omega \in S^1, t\geq 0\} = \mathbb{R}^2,$$
which finishes the proof.
\end{proof}

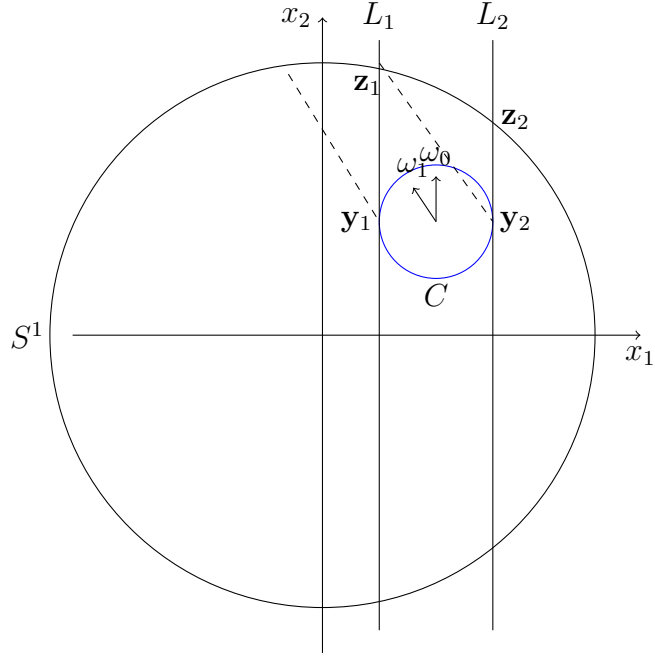
\begin{figure}
\centering
\begin{tikzpicture}[scale=3]
\draw [->] (0,-1.4)--(0,1.4)node[left]{$x_2$};
\draw [->] (-1.1,0)--(1.4,0)node[below]{$x_1$};
\draw (0,0) circle (1.2);
\draw [blue] (0.5,0.5) circle (0.25);
\draw (0.25,-1.3)--(0.25,1.3)node[above]{$L_1$};
\draw (0.75,-1.3)--(0.75,1.3)node[above]{$L_2$};
\node at (0.5,0.175) {$C$};
\node at (0.2,1.1) {$\vz_1$};
\node at (0.85,0.95) {$\vz_2$};
\node at (0.15,0.5) {$\vy_1$};
\node at (0.85,0.5) {$\vy_2$};
\draw [dashed] (0.25,1.2)--(0.75,0.5);
\draw [dashed] (-0.15,1.15)--(0.25,0.5);
\node at (-1.3,0) {$S^1$};
\draw [->] (0.5,0.5)--(0.5,0.7)node[above]{$\omega_0$};
\draw [->] (0.5,0.5)--(0.4,0.65)node[above]{$\omega_1$};
\end{tikzpicture}
\caption{Geometry of Theorem \ref{rec_mu}. In this case $\omega_0$ is parallel to the $x_2$ axis and $\omega_1$ is parallel to the dashed lines. $S^1$ is the larger circle and $C$ is the smaller blue circle.}
\label{fig_proof}
\end{figure}

We can now generalize this result to $n$-dimensions.

\begin{corollary}\label{corr_det_mu}
Let $\Omega \Subset B = \{|\vx| <1\} \subset \mathbb{R}^n$ have smooth boundary $A = \partial \Omega$ and let there exist $\vv_0 \in A$ and some open $\Lambda \subset S^{n-1}$ such that $\{\vv_0 + t\omega : t\geq 0\}$ misses $\text{supp}(\mu)$ for any $\omega \in \Lambda$. Let us further assume $\mu \in L^\infty(\mathbb{R}^n)$ is such that $\text{supp}(\mu) \subset B$ and that $\mu = 0$ restricted to the $\Omega$. Then if $D_0\mu(\vv,\omega) = 0$ for any $\vv\in A$ and $\omega \in S^{n-1}$, $\mu = 0$.
\end{corollary}

\begin{proof}
Let $D_0\mu(\vv,\omega) = 0$ for any $\vv\in A$ and $\omega \in S^{n-1}$. Similarly to the proof of Theorem \ref{rec_mu}, choose a sphere $S \subset \Omega$ such that $\vv_0 \in S$. Let $\xi_0$ be the normal to $A$ at $\vv_0$ and let $\omega_0 \in \Lambda$ be such that $\omega \notin T_{\vv_0} S$, i.e., so that $\omega$ is not parallel to the tangent plane to $S$ and $\vv_0$. We can do this since $\Lambda$ is an open subset of $S^{n-1}$. Let
$$\Lambda_\epsilon(\omega_0) = \{ \omega\in S^{n-1} : \omega \cdot \omega_0 > \cos \epsilon\}.$$
Then, for $\epsilon > 0$ small enough, $\{\vv_0 + \Lambda_\epsilon(\omega_0)\}$ misses $\text{supp}(\mu)$  by assumption. Let $\omega_0^{\perp} \in S^{n-1}$ be orthogonal to $\omega_0$ and let us define the 2-D plane
$$P = \{\vv_0 + t\omega_0 + s\omega_0^{\perp} : (t,s) \in \mathbb{R}^2\}.$$
Then, since $\omega_0$ is not parallel to $T_{\vv_0}S$, the intersection $P \cap S = \gamma$ is a 1-D smooth closed curve (a circle) embedded in $P$. We now work in $P$ and apply Theorem \ref{rec_mu}. Let $\tilde{\mu}(s,t)$ be the restriction of $\mu$ to $P$. Then, $D_0\tilde{\mu}(\vv,\omega) = 0$ for any $\vv \in \gamma$ and $\omega \in S^1$. Further, $\{\vv_0 + t\omega : t\geq 0\}$ misses $\text{supp}(\tilde{\mu})$ for any $\omega \in \Lambda_\epsilon(\omega_0) \cap P$. $\tilde{\mu}$ is also clearly zero on the interior of $\gamma$ and is supported on a unit ball. Therefore, the conditions of Theorem \ref{rec_mu} are satisfied and $\mu = 0$ on $P$. We now repeat this argument for all $\omega_0^{\perp}$ in the orthogonal compliment of $\omega_0$ to complete the proof.
\end{proof}

We now relate the above theorems to our application.

\begin{theorem}\label{det_mu}
Let the conditions of Theorem \ref{rec_f} be satisfied and let us assume that every half ray eminating from $\partial \Omega$ intersects $\s$. Further assume that there exists an open $\Lambda \subset S^{n-1}$ of non-zero measure and $\vv_0 \in \partial \Omega$ such that $\{\vv_0 + \Lambda\}$ misses $\text{supp}(\mu)$. Then $\mathcal{D}_{n-2}(f,\mu)$ determines $f$ and $\mu$ uniquely.
\end{theorem}
\begin{proof}
As the conditions of Theorem \ref{rec_f} are satisfied, $f$ is determined. In the discussions following Theorem \ref{rec_f}, we showed that once $f$ is recovered, we can determine $D_0\mu(\vv,\omega)$ (i.e., the linear transform of $\mu$) for any $\vv \in \partial \Omega$ and $\omega$ such that $L_{\vv,\omega}$ intersects $\s$. By our assumption that every half ray eminating from $\partial \Omega$ intersects $\s$, this means that we can determine $D_0\mu(\vv,\omega)$ for any $\vv \in \partial \Omega$ and $\omega \in S^{n-1}$. Now we aim to apply Corollary \ref{corr_det_mu} with $A = \partial \Omega$. By assumption, there exists an open $\Lambda \subset S^{n-1}$ and $\vv_0 \in A$ such that $\{\vv + \Lambda\}$ misses $\text{supp}(\mu)$, and $A$ is a smooth closed surface. Since we are also assuming that $\text{supp}(\mu) \cap \Omega_c = \emptyset$, $\mu = 0$ on the interior of $A$, and $\mu$ is compactly supported. We can therefore apply Corollary \ref{corr_det_mu} to show that $\mu$ is unique to the data.
\end{proof}

\begin{corollary} \label{cor:uniqueness}
Under the conditions of Theorem \ref{det_mu}, $\mathcal{R}(f,\mu)$ determines $f$ and $\mu$ uniquely.
\end{corollary}

\begin{proof}
By Corollary \ref{corr_1}, $\mathcal{R}(f,\mu)$ determines $\mathcal{D}_{n-2}(f,\mu)$. To get the result, we then apply Theorem \ref{det_mu}.
\end{proof}

\begin{remark}
An example where Theorem \ref{det_mu} and Corollary \ref{cor:uniqueness} apply is when $\s = S^{n-1}$ and $f$ and $\mu$ are supported on its interior, as pictured in figure \ref{fig1}. The condition we have used throughout that $\{\vv_0 + \Lambda\}$ misses $\text{supp}(\mu)$, while not generally true, would hold in general if the number of attenuating objects (the $\Omega_i$) are not too many or too large relative to the size of the source. For example, if $\s = S^{n-1}$, and there is only one attenuating domain $\Omega_1$, which is contained in a ball within $S^{n-1}$, then there is always an open set of rays eminating from $\partial \Omega$ which miss $\Omega_1$ as in the theorem conditions. If we assume that $\Omega$ and the $\Omega_i$ can be constrained to balls, then one could imagine using ideas from sphere packing to determine if there are rays from the source that do not attenuate. That is, is it possible to pack a set of $m$ balls of given radius $r$ which contain attenuating material around $\Omega$ so that all photons eminating from the source attenuate before reaching $\s$. Can we give bounds on $m$ and $r$ so the solution is unique? This is an interesting problem, but beyond the scope of this paper.
\end{remark}

While we have shown under certain conditions that $\mu$ is unique to the data, the problem of recovering $\mu$ is one of exterior tomography and is highly ill-posed, and thus we cannot expect to recover $\mu$ accurately. From a microlocal perspective, there are many missing wavefronts, and so we would expect the edges of $\mu$ to be blurred out and unresolved in the reconstruction without sufficient a-priori information. We can still use the values of $D_0\mu$ (calculated as described above) to attempt a reconstruction from limited ray integral data, which we will explore in the next section using algebraic methods.

\section{Simulations}
\label{results}
In this section, we conduct simulated experiments to help illustrate our theory. We consider the geometry of figure \ref{fig1}. We set $n=2$ and $\s = \{|\vx| = r\}$ is the circle radius $r = 128$ with center at the origin. The reconstrution space is the square region $[-r,r]^2$, and the image resolution is $2r \times 2r$. We consider reconstruction of the gamma ray intensity and attenuation phantoms shown in figure \ref{F1}. $f = I_0\chi_\Omega$ is the characteristic function on the non-convex region pictured and $\mu$ is a characteristic on the union of four balls. One can check easily that the chosen geometry and $f$ and $\mu$ satisfy the geometric conditions specified in Assumption \ref{geo_ass}.
\begin{figure}[!h]
\centering
\begin{subfigure}{0.27\textwidth}
\includegraphics[width=0.9\linewidth, height=3.5cm, keepaspectratio]{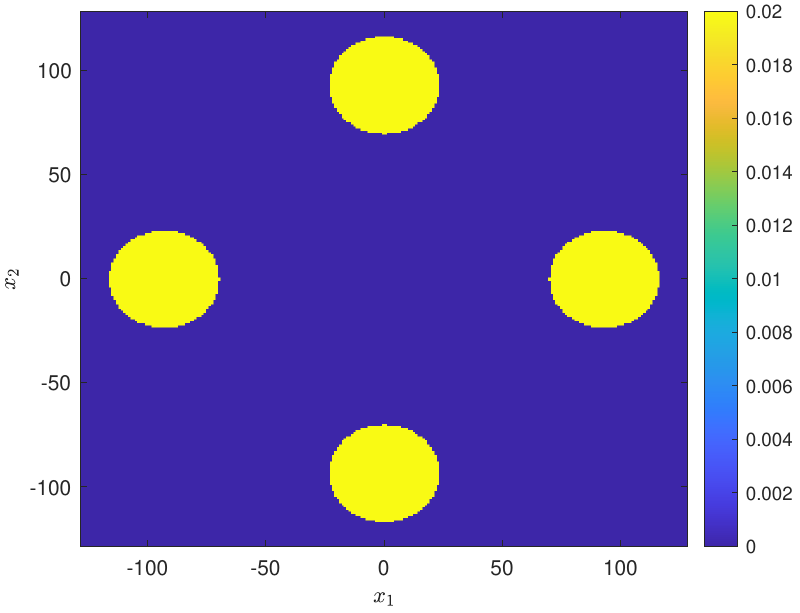}
\subcaption{$\mu$}
\end{subfigure}
\begin{subfigure}{0.27\textwidth}
\includegraphics[width=0.9\linewidth, height=3.5cm, keepaspectratio]{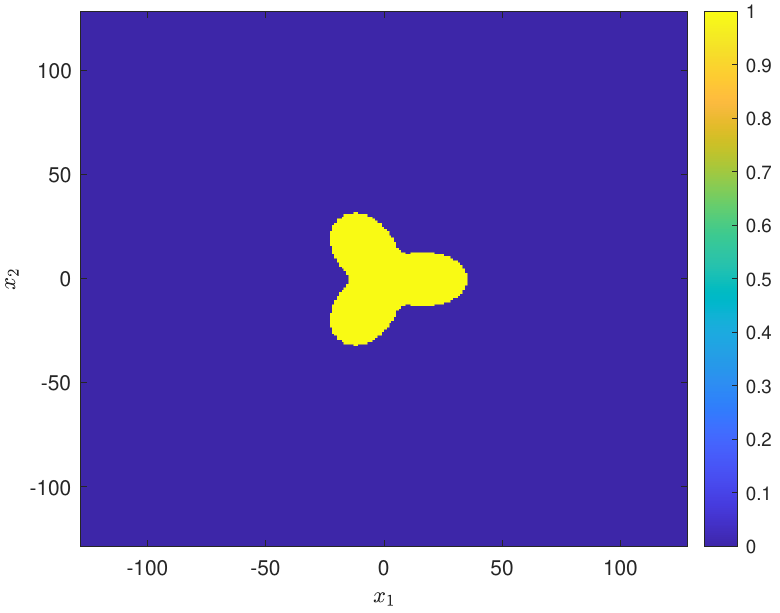}
\subcaption{$f$}
\end{subfigure}
\caption{Attenuation coefficient of surrounding material (left) and gamma ray source (right).}
\label{F1}
\end{figure}
The gamma ray source intensity value $I_0$ is 1 and the attenuation coefficient has value $1/50$. We simulate data, $\mathcal{R}(f,\mu)$, using \eqref{orig_data} with $\alpha = \pi/2$ (i.e., the opening angle between the V-lines is $\pi/2$) and add Gaussian noise with mean zero and standard deviation $\sigma$ to simulate measurement errors. From here, we calculate $g = \mathcal{D}_0(f,\mu)(\vv,\omega)$ using the algorithm described in Theorem \ref{simple_algo}. As noted in Remark \ref{remark_noise}, since $\alpha = \pi/2$, the corresponding noise on the $\mathcal{D}_0(f,\mu)$ data has standard deviation $\sqrt{2}\sigma$. 

We parameterize $\vv = r\Theta$, where $\Theta = (\cos\theta,\sin\theta)$ and $\theta$ is sampled uniformly on $[0,2\pi]$ at 360 steps. $\omega = \omega(t)$ is parameterized using $t \in \mathbb{R}$, where
$$\omega(t) = \frac{\omega'(t) - r\Theta}{|\omega'(t)-r\Theta|},\ \ \ \omega'(t) = -r\Theta + t\Theta^{\perp},$$
and $t \in [-402,402]$, sampled uniformly at 1025 steps. Here, $t$ parametrizes the plane tangent to $\s$ at $-r\Theta$. This is a classical fan-beam geometry for the 2-D Radon transform, where $r\Theta$ is the source point and $P = \{-r\Theta + t\Theta^{\perp} : t \in [-402,402]\}$ represents the detector plane. Throughout this section, we set $\sigma = 0.13$, which equates to $1\%$ relative least-squares error on the raw data. In figure \ref{F2}, we present sinograms of the simulated $\mathcal{D}_0(f,\mu)(\vv,\omega) = \mathcal{D}_0(f,\mu)(\theta,t)$ data.
\begin{figure}[!h]
\centering
\begin{subfigure}{0.27\textwidth}
\includegraphics[width=0.9\linewidth, height=3.5cm, keepaspectratio]{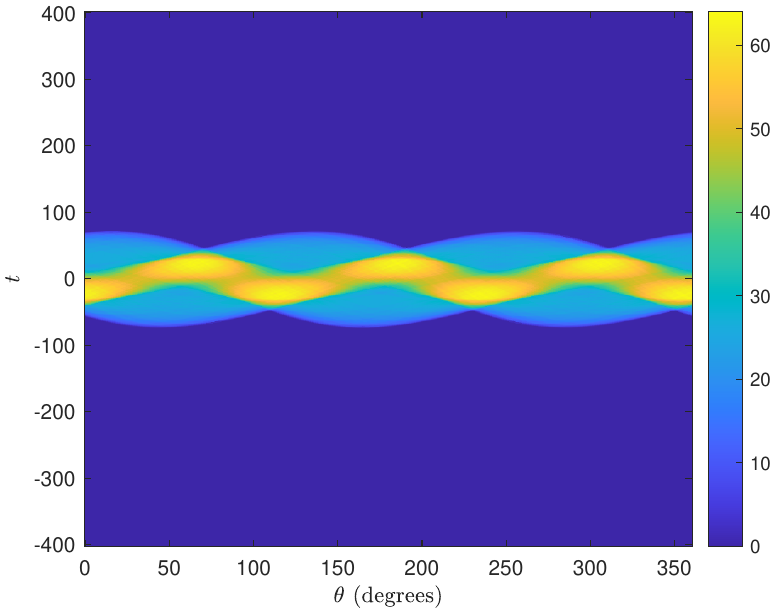}
\subcaption{$D_0f$}
\end{subfigure}
\begin{subfigure}{0.27\textwidth}
\includegraphics[width=0.9\linewidth, height=3.5cm, keepaspectratio]{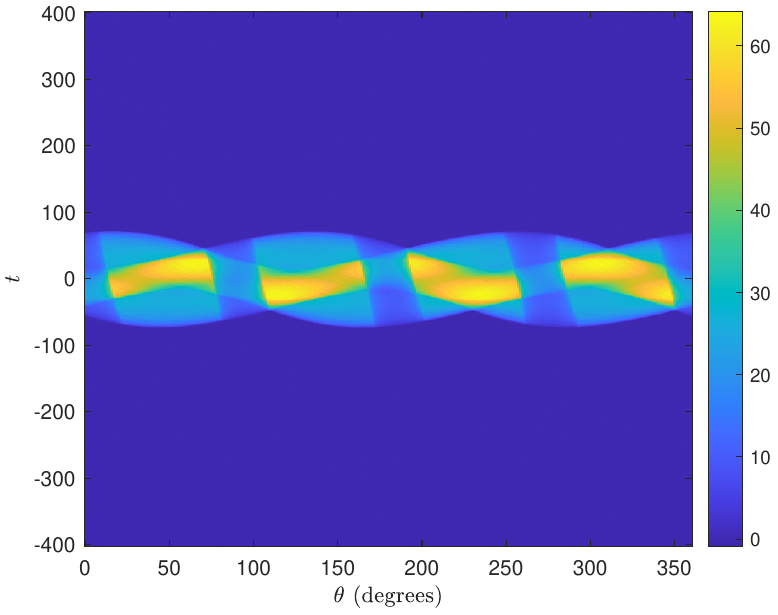}
\subcaption{$\mathcal{D}_0(f,\mu)$}\label{F2b}
\end{subfigure}
\caption{$\mathcal{D}_0(f,\mu)$ sinograms  with $1\%$ noise under the linear formulation, i.e., when $\mu = 0$ (left), and when $\mu$ is as in figure \ref{F1} (right).}
\label{F2}
\end{figure}
For comparison, we show $\mathcal{D}_0(f,\mu)$ side-by-side with the linear sinogram, i.e., $D_0f$. The linear sinogram is classical 2-D Radon transform data in the fan-beam geometry discussed above. When $\mu$ is non-zero and set as in figure \ref{F1}, much of the signal is suppressed in the sinogram due to attenuation, and the regions of the sinogram where the signal is lower are four-fold periodic in $\theta$ given the analogous rotational symmetry of $\mu$.
\subsection{Recovering the convex hull of $f$}
Now, we discuss how the convex hull, $\Omega_c$ of $f$ is recovered. To recover $\Omega_c$, we use Proposition \ref{prop_hull} and \eqref{equ_hull}. To use \eqref{equ_hull}, we need to determine $\Gamma_0$ (as defined in Proposition \ref{prop_hull}), which describes the zeros of the clean data $\mathcal{D}_0(f,\mu)$. We now describe a method to determine $\Gamma_0$ in the presence of noise using ideas from change-point detection.
\begin{figure}[!h]
\centering
\begin{subfigure}{0.4\textwidth}
\includegraphics[width=0.9\linewidth, height=4.5cm, keepaspectratio]{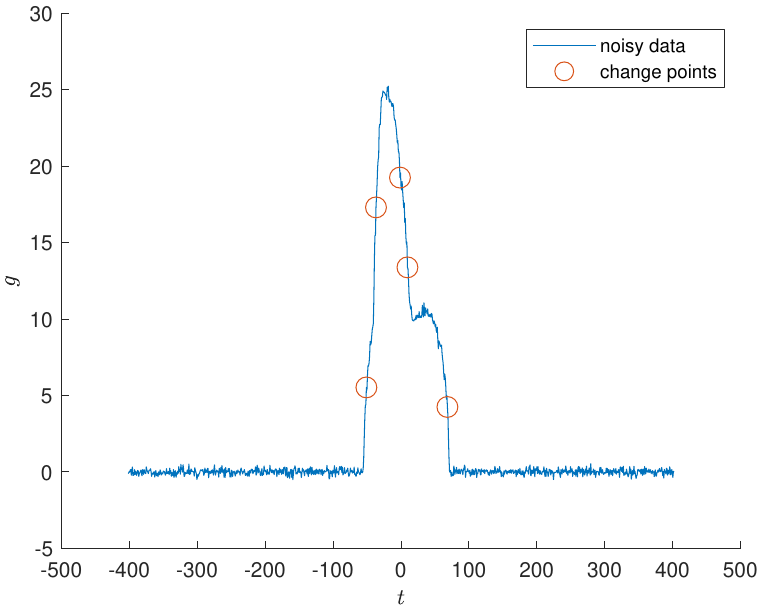}
\end{subfigure}
\caption{$\mathcal{D}_0(f,\mu)(0,\cdot)$ plot with change points highlighted.}
\label{F3}
\end{figure}

See figure \ref{F3}, where we have plotted $\mathcal{D}_0(f,\mu)(0, \cdot)$ (i.e., the first vertical line profile of the image in figure \ref{F2b}). We see large jumps in signal when the data has singularities. Before the first jump and after the last jump, the noiseless data is zero (in-between, the values are strictly greater than zero), and we use this idea to determine $\Gamma_0$. To calculate where the jumps occur in the presence of noise, we use the Matlab function ``findchangepts", setting the maximum number of change points to 5. The values before the first change point and after the last change point are determined to be the zeros of the data. We repeat this process for all line-profiles of the sinogram in figure \ref{F2b} to calculate $\Gamma_0$. After $\Gamma_0$ is determined, we use \eqref{equ_hull} to calculate $\Omega_c$.

\begin{figure}[!h]
\centering
\begin{subfigure}{0.27\textwidth}
\includegraphics[width=0.9\linewidth, height=3.5cm, keepaspectratio]{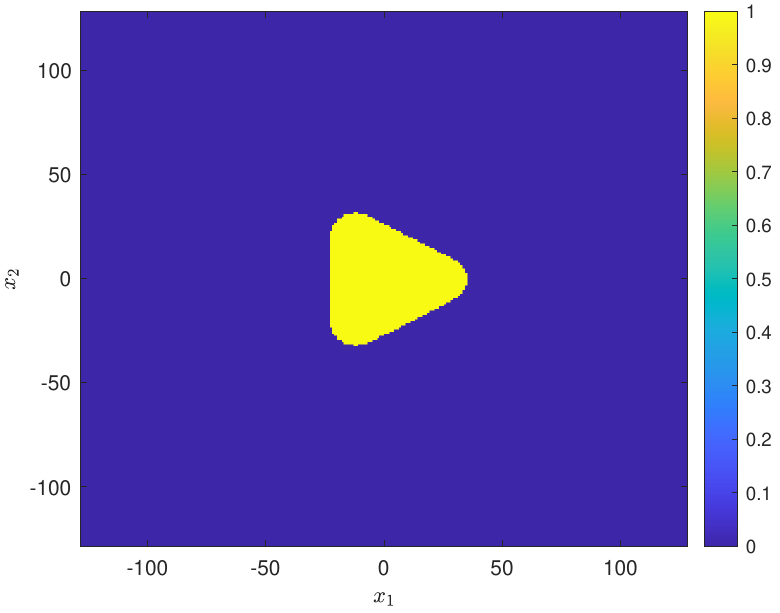}
\subcaption{$\Omega_c$}
\end{subfigure}
\begin{subfigure}{0.27\textwidth}
\includegraphics[width=0.9\linewidth, height=3.5cm, keepaspectratio]{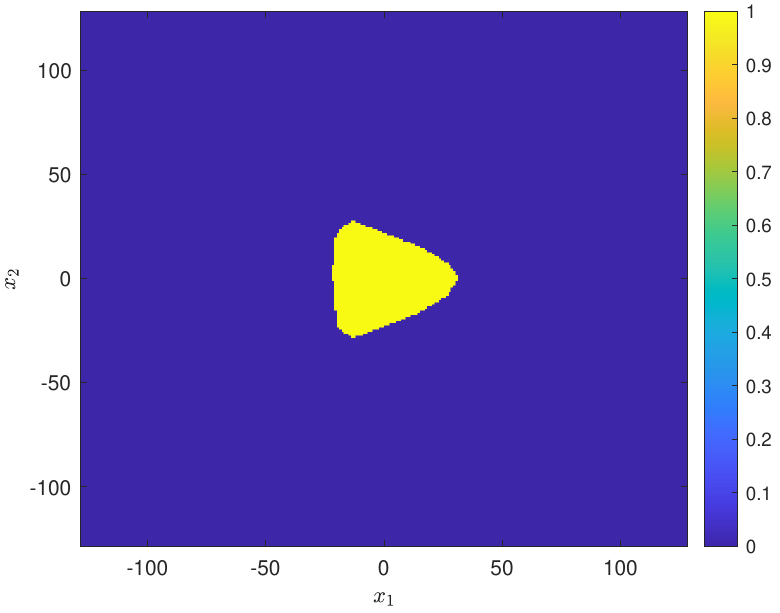}
\subcaption{reconstruction}\label{F4b}
\end{subfigure}
\begin{subfigure}{0.27\textwidth}
\includegraphics[width=0.9\linewidth, height=3.5cm, keepaspectratio]{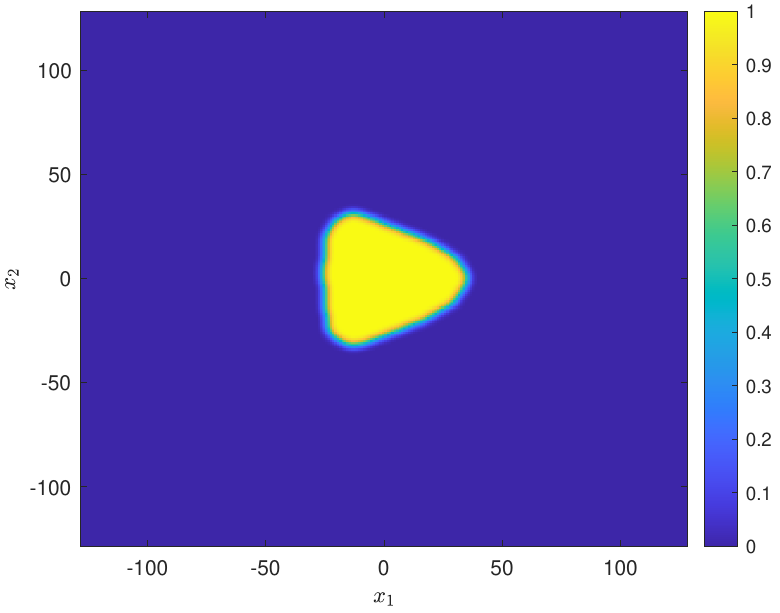}
\subcaption{$\psi$}\label{F4c}
\end{subfigure}
\caption{A reconstruction of $\Omega_c$. We show the exact $\Omega_c$ (left), which was calculated using noiseless data, the reconstruction (middle), and a smooth cutoff $\psi$ with support on a slightly larger set than the reconstructed $\Omega_c$.}
\label{F4}
\end{figure}

In figure \ref{F4}, we show a reconstruction of $\Omega_c$ using the noisy sinogram in figure \ref{F2b} and compare to the ``exact" convex hull calculated using noiseless data. In figure \ref{F4c}, we convolve the reconstruction in figure \ref{F4b} with a Gaussian filter to get a smooth cutoff, $\psi$, supported on a slightly larger set than $\Omega_c$. This cutoff will be used as in Corollary \ref{corr_bound} to remove artifacts from the reconstruction of $f$, which we will now discuss in more detail.

\subsection{Recovery of $f$}
Here, we recover $f$ using the method detailed in section \ref{sing_recon}. As in section \ref{sing_recon}, $f$ is reconstructed via the formula
$$f_r = \psi X^{-1}g,$$
where $\psi$ is as in figure \ref{F4c}, $g = \mathcal{D}_0(f,\mu)$ is our data, and $X^{-1}$ in this case is the 2-D inverse Radon transform. To apply $X^{-1}$, we first discretize the corresponding linearized forward operators and apply the Landweber method \cite[Section 5.1]{bertero2021introduction}.
\begin{figure}[!h]
\centering
\begin{subfigure}{0.27\textwidth}
\includegraphics[width=0.9\linewidth, height=3.5cm, keepaspectratio]{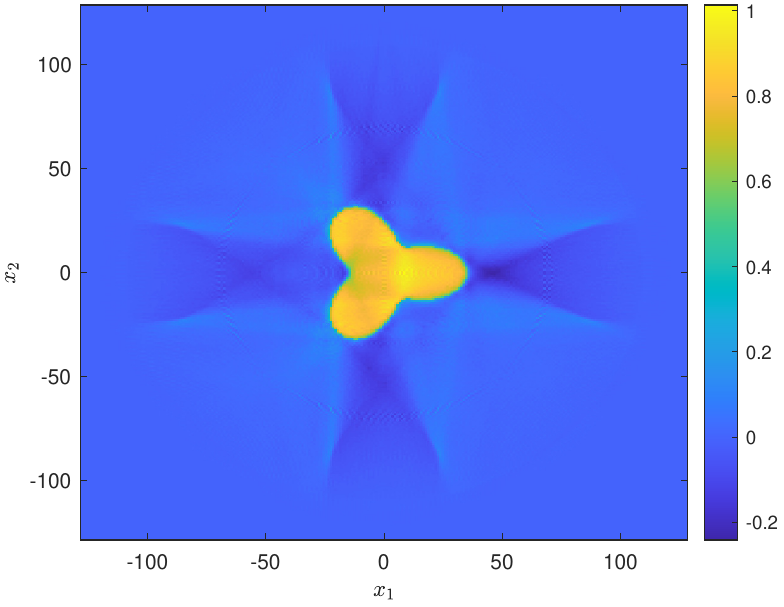}
\subcaption{$X^{-1}g$}\label{F5a}
\end{subfigure}
\begin{subfigure}{0.27\textwidth}
\includegraphics[width=0.9\linewidth, height=3.5cm, keepaspectratio]{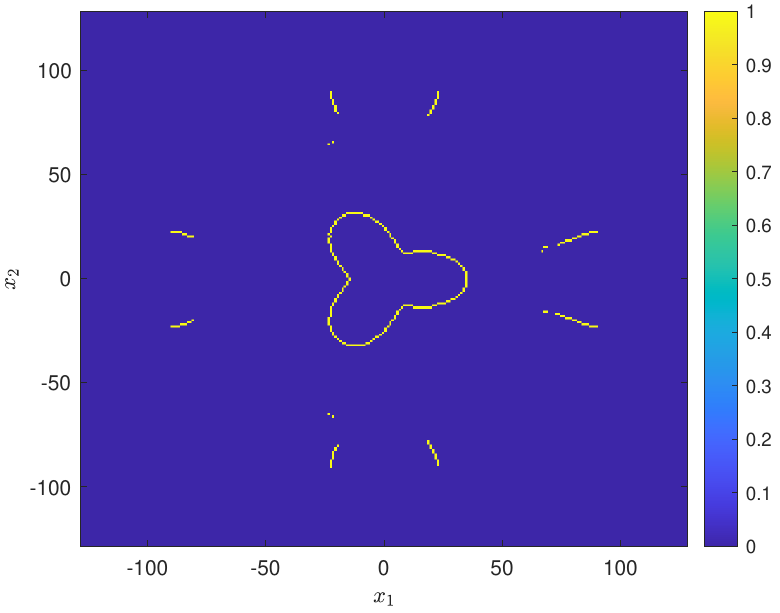}
\subcaption{edges (artifacts)}\label{F5b}
\end{subfigure}
\begin{subfigure}{0.27\textwidth}
\includegraphics[width=0.9\linewidth, height=3.5cm, keepaspectratio]{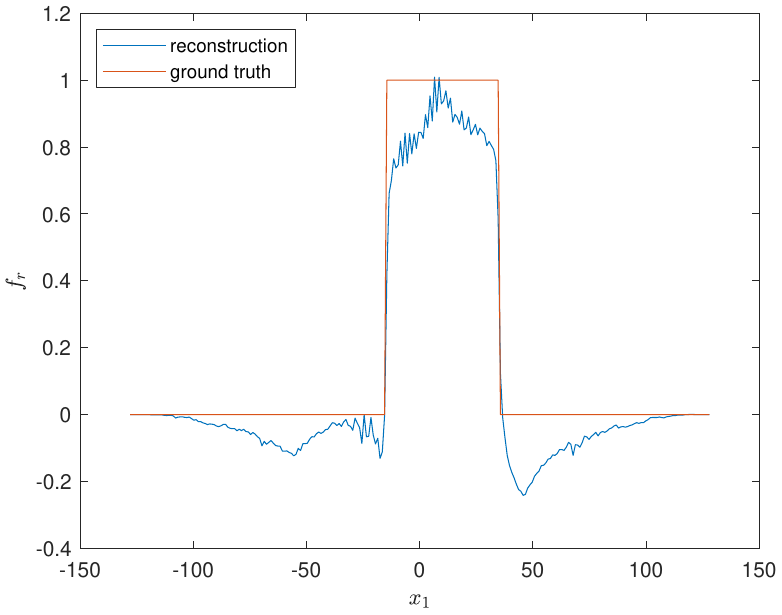}
\subcaption{line-profile of (A)}
\end{subfigure}
\begin{subfigure}{0.27\textwidth}
\includegraphics[width=0.9\linewidth, height=3.5cm, keepaspectratio]{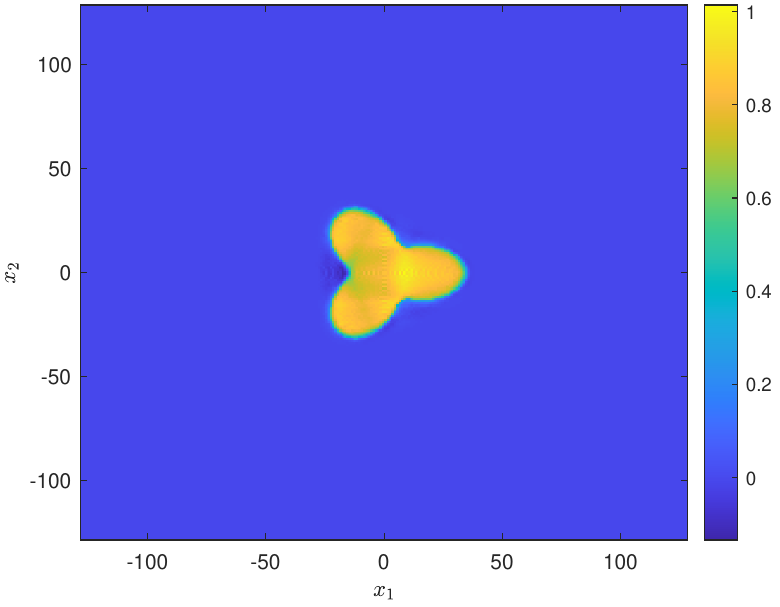}
\subcaption{$f_r = \psi X^{-1}g$}
\end{subfigure}
\begin{subfigure}{0.27\textwidth}
\includegraphics[width=0.9\linewidth, height=3.5cm, keepaspectratio]{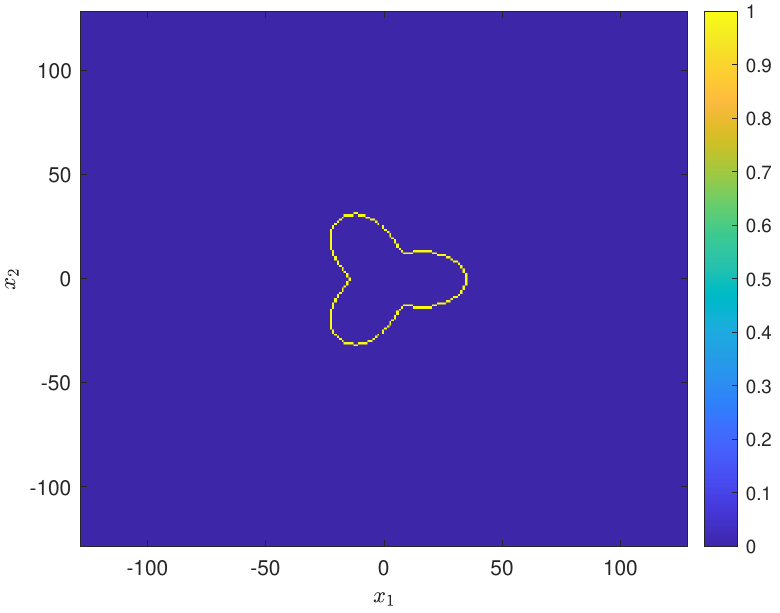}
\subcaption{edges (no artifacts)}\label{F5e}
\end{subfigure}
\begin{subfigure}{0.27\textwidth}
\includegraphics[width=0.9\linewidth, height=3.5cm, keepaspectratio]{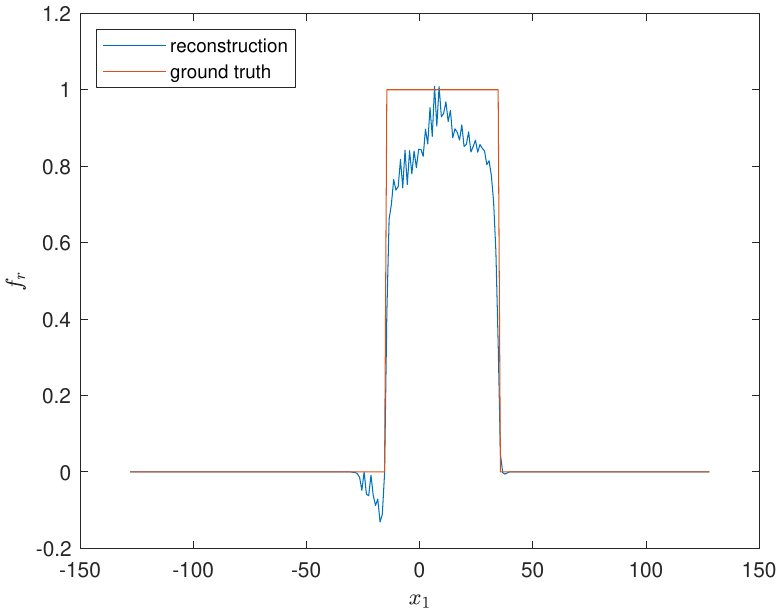}
\subcaption{line-profile of (D)}
\end{subfigure}
\caption{Reconstructions of $f$. The top row corresponds to $X^{-1}g$ and the bottom row to $f_r$. We show the reconstruction (left), an edge map extraction (middle), and the central line-profile along $\{x_2=0\}$ (right).}
\label{F5}
\end{figure}

In figure \ref{F5}, we present our reconstructions. For comparison, we show $f_r$ alongside $X^{-1}g$ (i.e., without multiplying by $\psi$) to show the artifacts and how they are removed. We also estimate the edges of the reconstruction using the Matlab function ``edge" with a Sobel filter and the default settings. In figure \ref{F5a}, we see streaking artifacts in the reconstruction which occur along rays which are tangent to $\mu$ and $f$ simultaneously. In the edge map in figure \ref{F5b}, we see $\partial\Omega$ (i.e., the desired edges of $f$) and the parts of $\Sigma = \text{ssupp}(\mu)$ which are visible in the data. These appear as circular arc segments in the edge map. In Theorem \ref{double} and the following discussions in section \ref{sing_recon}, we predicted that such streaking artifacts would occur and that the strength of the artifacts would be weaker than the visible singularities of $f$ and $\mu$. The edge map extractions find the strongest edges above a certain threshold, and thus it makes sense that the streaks do not appear in figure \ref{F5b}. On the bottom row of figure \ref{F5}, we see the artifacts are largely removed as predicted, and we can recover $\partial \Omega$ as in figure \ref{F5e}.

Now, we are in a position to reconstruct $f = I_0 \chi_\Omega$. To recover $\Omega$, we can simply fill-in the boundary recovered in figure \ref{F5e}. To determine $I_0$, we use the method provided in Theorem \ref{rec_f}. We calculated $I_0$ to be $I_0 = 1.06$ (i.e., $6\%$ error from the true value of 1) in this example. 
\begin{figure}[!h]
\centering
\begin{subfigure}{0.27\textwidth}
\includegraphics[width=0.9\linewidth, height=3.5cm, keepaspectratio]{source_phan}
\subcaption{$f$}
\end{subfigure}
\begin{subfigure}{0.27\textwidth}
\includegraphics[width=0.9\linewidth, height=3.5cm, keepaspectratio]{r2}
\subcaption{$f_r$ ($\eta = 26\%$)}
\end{subfigure}
\begin{subfigure}{0.27\textwidth}
\includegraphics[width=0.9\linewidth, height=3.5cm, keepaspectratio]{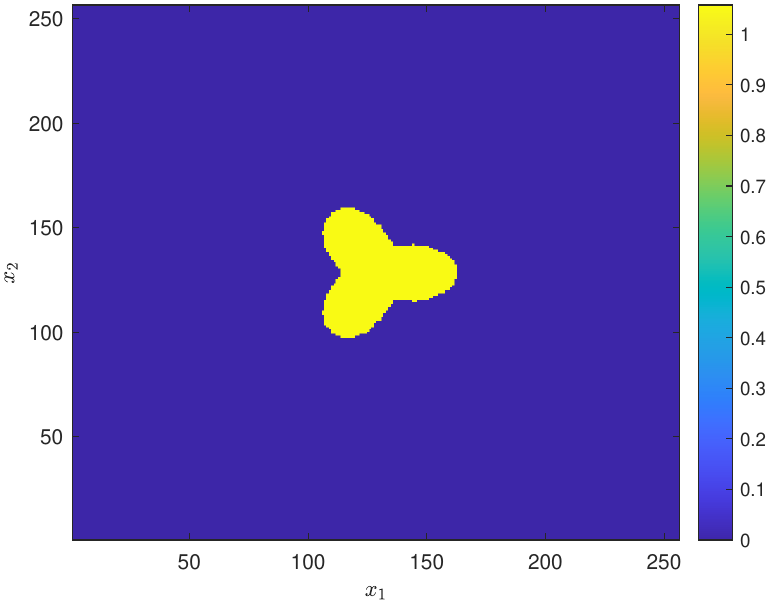}
\subcaption{$f_c$ ($\eta = 21\%$)}\label{F6c}
\end{subfigure}
\caption{The ground truth $f$ (left), $f_r = \psi X^{-1}g$ (middle) and $f_c = I_0 \chi_\Omega$ (right) using the reconstructed $\Omega$ and $I_0$ value. The relative least-squares error values ($\eta$) are given in parenthesis.}
\label{F6}
\end{figure}
See figure \ref{F6}, where we show the ground truth $f$ side-by-side with $f_r = \psi X^{-1}g$ and $f_c = I_0 \chi_\Omega$ using the reconstructed $\Omega$ and $I_0$ value. These are two alternate means of reconstructing $f$, the latter of which uses the a-priori knowledge that $f$ is proportional to a characteristic. Both $f_r$ and $f_c$ accurately identify the location, shape and size of $f$, and $f_c$ (i.e., the reconstruction in figure \ref{F6c}) is the most accurate in terms of least-squares error.

\subsection{Recovery of $\mu$}
In this section, we recover $\mu$ using the data $h = X_{w_1}\mu$ as in \eqref{recov_mu}, where $h$ is determined using the raw data $g$, and our reconstructed $f$ as in \eqref{recov_mu} and the following discussions. In this example, $h$ describes all integrals of $\mu$ over divergent beams eminating from $\partial \Omega$. As above, for reconstruction, we first discretize $X_{w_1}$ and apply the Landweber method to approximate the inverse. After the reconstruction is obtained, we multiply the image by a smooth cutoff which is zero on $\Omega$. This is to remove artifacts which appear near $\partial \Omega$, i.e., the end points of the divergent beams used for reconstruction. Our reconstruction  is presented in figure \ref{F7}.
\begin{figure}[t]
\centering
\begin{subfigure}{0.27\textwidth}
\includegraphics[width=0.9\linewidth, height=3.5cm, keepaspectratio]{att_phan}
\subcaption{$\mu$}
\end{subfigure}
\begin{subfigure}{0.27\textwidth}
\includegraphics[width=0.9\linewidth, height=3.5cm, keepaspectratio]{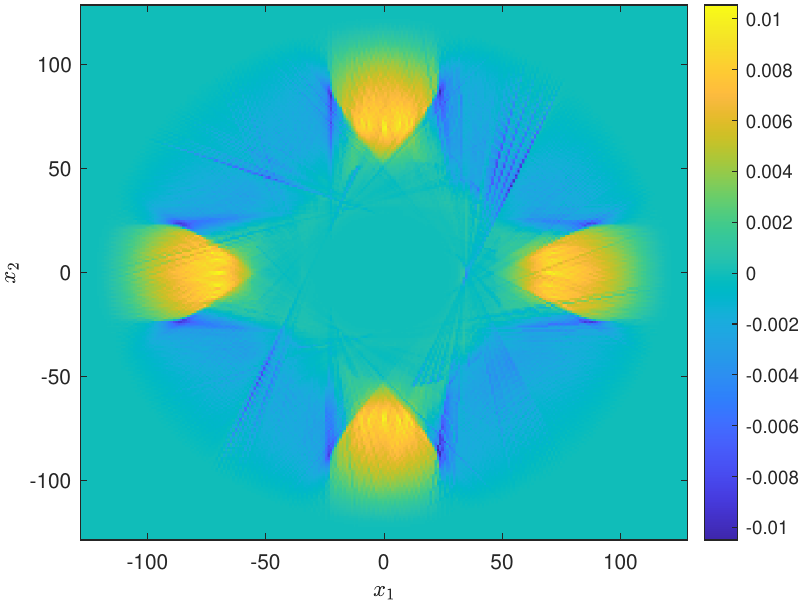}
\subcaption{$\mu_r$}\label{F6b}
\end{subfigure}
\caption{The ground truth $\mu$ and a reconstruction, $\mu_r$, using the data $h$ and the Landweber method.}
\label{F7}
\end{figure}
The reconstruction quality is poor, and only the edges of $\mu$ shown in figure \ref{F5b} (i.e., those visible to the data) are resolved. This is to be expected as the inverse problem is one of exterior tomography and is severely ill-posed. The artifacts observed (e.g., the blurring or ``smearing" of the true edges of $\mu$) are akin to those seen in classical limited angle tomography \cite{frikel2013characterization}.

\section{Conclusion}
This work advances microlocal analysis of operators with non-linearities, which is directly relevant to Compton camera imaging. We notably addressed the case when $\mu$ can have singularities, and described how this led to artifacts in the reconstruction. We used our geometric assumptions to better localize the edges of $f$ and separate these from the edges of $\mu$ and the artifacts. This allowed us to recover $f$ accurately. While we showed that $\mu$ is theoretically recoverable, the resulting reconstruction was not good due to limited angle effects. In this work, the cone vertices were restricted to an $(n-1)$-D surface, and the data was overdetermined. In further work, we aim to consider situations where we don't have so much data (e.g., where the vertices lie on a curve), and investigate any additional artifacts or instabilities which may occur.

\section*{Acknowledgements:} 
The first author wishes to acknowledge funding support from Aspira Women's Health, The Cleveland Clinic Foundation, The Honorable Tina Brozman Foundation, the V Foundation, and the National Cancer Institute R03CA283252-01. Sean Holman was supported by Engineering and Physical Sciences Research Council (EPSRC) through grant EP/V007742/1. {We also wish to thank The Institute for Computational and Experimental Research in Mathematics (ICERM) at Brown University for hosting us for a
workshop in August 2024, which facilitated compelling discussions and advancements in
this project.}

\bibliographystyle{abbrv} 
\bibliography{RefRevolution}

\begin{thebibliography}{10}

\bibitem{ambartsoumian2012inversion}
G.~Ambartsoumian.
\newblock Inversion of the {V}-line {R}adon transform in a disc and its applications in imaging.
\newblock {\em Computers \& Mathematics with Applications}, 64(3):260--265, 2012.

\bibitem{AmbartsoumianLatifi-Vline2019}
G.~Ambartsoumian and M.~J. Latifi~Jebelli.
\newblock The {V}-line transform with some generalizations and cone differentiation.
\newblock {\em Inverse Problems}, 35(3):034003, 29, 2019.

\bibitem{bertero2021introduction}
M.~Bertero, P.~Boccacci, and C.~De~Mol.
\newblock {\em Introduction to inverse problems in imaging}.
\newblock CRC press, 2021.

\bibitem{vline}
J.~Cebeiro, M.~A. Morvidone, and M.~K. Nguyen.
\newblock {The {R}adon transform on {V}-lines: Artifact analysis and image enhancement.}
\newblock In {\em XVII Workshop on Information Processing and Control (RPIC)}, pages 1--6. IEEE, 2017.

\bibitem{chihara2026geodesic}
H.~Chihara.
\newblock Geodesic {X}-ray transform and streaking artifacts on simple surfaces or on spaces of constant curvature: H. chihara.
\newblock {\em Vietnam Journal of Mathematics}, pages 1--50, 2026.

\bibitem{duistermaat1996fourier}
J.~J. Duistermaat and L.~Hormander.
\newblock {\em {F}ourier integral operators}, volume~2.
\newblock Springer, 1996.

\bibitem{franssens2014multiplication}
G.~R. Franssens.
\newblock Multiplication of the distributions $(x\pm i0)^z$.
\newblock {\em Journal of Applied Analysis}, 20(1):15--27, 2014.

\bibitem{frikel2013characterization}
J.~Frikel and E.~T. Quinto.
\newblock Characterization and reduction of artifacts in limited angle tomography.
\newblock {\em Inverse Problems}, 29(12):125007, 2013.

\bibitem{greenleaf1993recovering}
A.~Greenleaf and G.~Uhlmann.
\newblock Recovering singularities of a potential from singularities of scattering data.
\newblock {\em Communications in mathematical physics}, 157(3):549--572, 1993.

\bibitem{holman2020spect}
S.~Holman and P.~Richardson.
\newblock {SPECT} with a multi-bang assumption on attenuation.
\newblock {\em Inverse Problems}, 36(12):125005, 2020.

\bibitem{SeanHolman}
S.~Holman and P.~Richardson.
\newblock Simultaneous recovery of attenuation and source density in {SPECT}.
\newblock {\em Inverse Problems and Imaging}, 17(4):817--840, 2023.

\bibitem{hormanderI}
L.~H{\"o}rmander.
\newblock {\em The analysis of linear partial differential operators. {I}}.
\newblock Classics in Mathematics. Springer-Verlag, Berlin, 2003.
\newblock Distribution theory and {F}ourier analysis, Reprint of the second (1990) edition [Springer, Berlin].

\bibitem{hormanderIII}
L.~H\"{o}rmander.
\newblock {\em The analysis of linear partial differential operators. {III}}.
\newblock Classics in Mathematics. Springer, Berlin, 2007.
\newblock Pseudo-differential operators, Reprint of the 1994 edition.

\bibitem{hormander}
L.~H\"{o}rmander.
\newblock {\em The analysis of linear partial differential operators. {IV}}.
\newblock Classics in Mathematics. Springer-Verlag, Berlin, 2009.
\newblock Fourier integral operators, Reprint of the 1994 edition.

\bibitem{katsevich2026analysis}
A.~Katsevich.
\newblock Analysis of beam hardening streaks in tomography.
\newblock {\em Inverse Problems}, 42(2):025003, 2026.

\bibitem{krishnan2014microlocal}
V.~P. Krishnan and E.~T. Quinto.
\newblock Microlocal analysis in tomography.
\newblock {\em Handbook of mathematical methods in imaging}, pages 1--50, 2014.

\bibitem{kuchment2016three}
P.~Kuchment and F.~Terzioglu.
\newblock Three-dimensional image reconstruction from compton camera data.
\newblock {\em SIAM Journal on Imaging Sciences}, 9(4):1708--1725, 2016.

\bibitem{kuchment2017inversion}
P.~Kuchment and F.~Terzioglu.
\newblock Inversion of weighted divergent beam and cone transforms.
\newblock {\em Inverse Problems and Imaging}, 11(6):1071--1090, 2017.

\bibitem{lee2018introduction}
J.~M. Lee.
\newblock {\em Introduction to Riemannian manifolds}, volume~2.
\newblock Springer, 2018.

\bibitem{monard2015geodesic}
F.~Monard, P.~Stefanov, and G.~Uhlmann.
\newblock The geodesic ray transform on {R}iemannian surfaces with conjugate points.
\newblock {\em Communications in Mathematical Physics}, 337(3):1491--1513, 2015.

\bibitem{moon2025inversion}
S.~Moon and M.~Haltmeier.
\newblock Inversion formulas for the attenuated conical {R}adon transform: Plane and cylinder case.
\newblock {\em Applied Mathematics and Computation}, 489:129159, 2025.

\bibitem{morvidone2010v}
M.~Morvidone, M.~K. Nguyen, T.~T. Truong, and H.~Zaidi.
\newblock On the {V}-line {R}adon transform and its imaging applications.
\newblock {\em International Journal of Biomedical Imaging}, 2010, 2010.

\bibitem{natterer1981identification}
F.~Natterer.
\newblock The identification problem in emission computed tomography.
\newblock In {\em Mathematical Aspects of Computerized Tomography: Proceedings, Oberwolfach, February 10--16, 1980}, pages 45--56. Springer, 1981.

\bibitem{natterer}
F.~Natterer.
\newblock {\em {The mathematics of computerized tomography}}.
\newblock Classics in Mathematics. Society for Industrial and Applied Mathematics (SIAM), New York, 2001.

\bibitem{palacios2018quantitative}
B.~Palacios, G.~Uhlmann, and Y.~Wang.
\newblock Quantitative analysis of metal artifacts in x-ray tomography.
\newblock {\em SIAM Journal on Mathematical Analysis}, 50(5):4914--4936, 2018.

\bibitem{Q1983-rotation}
E.~T. Quinto.
\newblock {The invertibility of rotation invariant Radon transforms}.
\newblock {\em J. Math. Anal. Appl.}, 94:602--603, 1983.

\bibitem{Q1993sing}
E.~T. Quinto.
\newblock {Singularities of the {X}-ray transform and limited data tomography in ${\mathbb R}^2$ and ${\mathbb R}^3$}.
\newblock {\em SIAM J. Math. Anal.}, 24:1215--1225, 1993.

\bibitem{Rudin:FA}
W.~Rudin.
\newblock {\em Functional analysis}.
\newblock McGraw-Hill Book Co., New York, 1973.
\newblock McGraw-Hill Series in Higher Mathematics.

\bibitem{solmon1995identification}
D.~C. Solmon.
\newblock The identification problem for the exponential {R}adon transform.
\newblock {\em Mathematical methods in the applied sciences}, 18(9):687--695, 1995.

\bibitem{stefanov2014identification}
P.~Stefanov.
\newblock The identification problem for the attenuated {X}-ray transform.
\newblock {\em American Journal of Mathematics}, 136(5):1215--1247, 2014.

\bibitem{terzioglu2019some}
F.~Terzioglu.
\newblock Some analytic properties of the cone transform.
\newblock {\em Inverse Problems}, 35(3):034002, 2019.

\bibitem{terzioglu2020exact}
F.~Terzioglu.
\newblock Exact inversion of an integral transform arising in {C}ompton camera imaging.
\newblock {\em Journal of Medical Imaging}, 7(3):032504--032504, 2020.

\bibitem{terzioglu2018compton}
F.~Terzioglu, P.~Kuchment, and L.~Kunyansky.
\newblock Compton camera imaging and the cone transform: a brief overview.
\newblock {\em Inverse Problems}, 34(5):054002, 2018.

\bibitem{Webber2026Microlocal}
J.~Webber and S.~Holman.
\newblock Microlocal analysis of non-linear operators arising in compton ct.
\newblock {\em Inverse Problems}, 42(2), 2026.

\bibitem{yao2022rapid}
Z.~Yao, Y.~Yuan, J.~Wu, X.~Liu, and Y.~Xiao.
\newblock Rapid compton camera imaging for source terms investigation in the nuclear decommissioning with a subset-driven origin ensemble algorithm.
\newblock {\em Radiation Physics and Chemistry}, 197:110133, 2022.

\bibitem{zhang2020recovery}
Y.~Zhang.
\newblock Recovery of singularities for the weighted cone transform appearing in {C}ompton camera imaging.
\newblock {\em Inverse Problems}, 36(2):025014, 2020.

\end{thebibliography}


\end{document}